%% file: main.tex
\documentclass[11pt]{article} 
\usepackage{amsmath, amssymb, amsthm}
\usepackage{ytableau}
\usepackage{graphicx}
\usepackage{enumitem}
\usepackage{appendix}
\usepackage{tikz}
\usepackage{makecell}
\usepackage[all]{xy}
\usepackage[numbers,sort&compress]{natbib} 
\usepackage{cleveref}

\theoremstyle{plain}
\newtheorem{theorem}{Theorem}[section]
\newtheorem{proposition}[theorem]{Proposition}
\newtheorem{corollary}[theorem]{Corollary}
\newtheorem{lemma}[theorem]{Lemma}
\theoremstyle{definition}
\newtheorem{remark}[theorem]{Remark}
\newtheorem{definition}[theorem]{Definition}
\newtheorem{example}[theorem]{Example}

\renewenvironment{proof}{\noindent{\bf Proof.}}{\hfill$\square$\par}

\title{\textbf{Rainbow Boomerang Graphs}}

\author{
\textbf{Shunsuke Hirota}\\
\textbf{Department of Mathematics, Kyoto University}\\
Kitashirakawa Oiwake-cho, Sakyo-ku, 606-8502, Kyoto, Japan\\
\textit{E-mail address}: hirota.shunsuke.48s@st.kyoto-u.ac.jp
}

\begin{document}

\maketitle

\begin{abstract}
\input{ab}
\end{abstract}

\tableofcontents

\section{Introduction}
\input{1.1}

\subsection{Acknowledgements}

\input{1.6}

\section{Exchange property of edge-colored graphs}
\input{graph}

\section{Weyl groupoids} \label{app:Weyl_groupoids}
\input{A.1}

\input{A.2}

\section{Exchange property of odd reflections}

\input{A.3}

\section{Odd Verma's theorem for Nichols algebras of diagonal type} \label{app:Nichols_algebras}
\input{A.4}

\input{A.5}

\input{A.6}

\bibliographystyle{plainnat}
\bibliography{references}

\end{document}

%% file: ab.tex
\begin{abstract}
We generalize the well-known exchange property of Coxeter groups to the setting of edge-colored graphs.

This work aims to unify and extend the results of our companion article, “Odd Verma’s Theorem,” which were originally established for basic Lie superalgebras, to the broader setting of regular symmetrizable Kac–Moody Lie superalgebras and Nichols algebras of diagonal type, via the theory of Weyl groupoids in the sense of Heckenberger and Yamane. In particular, we show that the exchange property of odd reflections arises as a special case of the exchange property of Weyl groupoids.

To study the exchange property itself, we analyze a class of edge-colored graphs introduced here—called rainbow boomerang graphs—which form an independently natural family of combinatorial objects.

We also elaborate on odd Verma's theorem in the specific setting of Nichols algebras of diagonal type.
\end{abstract}

%% file: 1.1.tex
\subsection{Background and motivation}

The following is what should be regarded as the well-known exchange property for Coxeter groups.  

\begin{theorem}[{\cite[Corollary 1.7]{humphreys1992reflection}}]\label{g.lwnw}
Let \( \Delta \) be a (not necessarily crystallographic) root system, and let \( W \) be the associated Coxeter group. Fix a positive root system \( \Delta^+ \). Denote by \( l(w) \) the length of \( w \in W \), and let \( n(w) := \#(\Delta^+ \cap (-w(\Delta^+))) \). Then, 
\[
l(w) = n(w).
\]
\end{theorem}

We interpret this property as a structural feature of the (edge colored) Cayley graph of a Coxeter group, which naturally leads to a formulation in terms of walks on edge-colored graphs.  
Specifically, the invariant \( \ell(w) \) corresponds to being \textit{shortest}, while \( n(w) \) corresponds to being \textit{rainbow}.
Here, a \textit{rainbow } refers to a path in which all edges are assigned distinct colors, a notion widely studied in graph theory, particularly in the context of \textit{rainbow connectivity} \cite{li2013rainbow}.

In light of this, we define the exchange property for edge-colored graphs as follows.

   \begin{definition}\label{exchange}
    A connected properly edge-colored graph is called a \textit{rainbow boomerang graph} if it satisfies the following condition: A walk is shortest if and only if it is rainbow.
\end{definition}

The rainbow boomerang graph  naturally include important examples such as Young lattices and hypercubes.  The study of the rainbow boomerang graph carried out in Section~2 can be regarded as a study of the exchange property itself. Rainbow boomerang graphs can also be characterized purely in terms of the underlying graph without reference to edge colorings \cref{ko}, and is thus expected to be of independent natural from the viewpoint of pure graph theory.

The original motivation to study rainbow boomerang graphs stems from "odd Verma's theorem" established in the companion article~\cite{Hirota2025}, which required considering the exchange property for a quotient graph of the groupoid of odd reflections. In \cite{Hirota2025}, the main results were established for basic Lie superalgebras, where verifying the exchange property of these quotient graphs could be done on a case-by-case without much difficulty. However, in order to extend these results to settings such as regular symmetrizable Kac–Moody Lie superalgebras \cite{serganova2011kac,bonfert2024weyl} and Nichols algebras of diagonal type \cite{andruskiewitsch2017finite}, it becomes both natural and necessary to develop a unified framework for treating the exchange property systematically.

The Weyl groupoids (a term we occasionally use in place of “generalized root systems,” following common conventions) in the sense of Heckenberger and Yamane \cite{heckenberger2008generalization,heckenberger2020hopf,cuntz2009weylthree}, which were introduced in the context of the classification of Nichols algebras of diagonal type \cite{heckenberger2009classification,andruskiewitsch2017finite}, are recognized as a natural generalization of Conway-Coxeter's frieze patterns \cite{cuntz2009weyltwo,cuntz2014frieze,heckenberger2020hopf} (\cref{w.5}) and the root systems of (modular) regular symmetrizable Kac-Moody Lie superalgebras \cite{azam2015classification,bonfert2024weyl,bouarroudj2009classification,kac_weisfeiler_1971}. Remarkably, the Weyl groupoid also satisfies an exchange property (\cref{poid_lwnw}) that generalizes \cref{g.lwnw}.

In Section~3, we briefly explain, in terms of the rainbow boomerang graph, how the exchange property of the Weyl groupoid is inherited by a newly introduced class of substructures called path subgroupoids.
In Section~4, we explain that path subgroupoids generalize the groupoid of odd reflections for regular symmetrizable Kac--Moody Lie superalgebras, by appealing to results of \cite{bonfert2024weyl}. This allows us to understand the exchange property of odd reflections, as established in~\cite{gorelik2022root}, as a special case of the exchange property of Weyl groupoids.

As a consequence, odd Verma's theorem for regular symmetrizable Kac--Moody Lie superalgebras holds in exactly the same form as in~\cite{Hirota2025}, provided one carefully defines the Weyl vector in this setting. Odd Verma's theorem also extends to Nichols algebras of diagonal type in light of the exchange  property of Weyl groupoids; however, its appearance differs slightly from that in~\cite{Hirota2025}. For this reason, we briefly explain the adaptation of~\cite{Hirota2025} to this setting in Section~5.

In the work of Gorelik--Hinich--Serganova~\cite{gorelik2022root,gorelik2023matsumoto}, the exchange property is also discussed in the context of a class even broader than the BKM Lie superalgebras introduced therein. They also remark that the groupoid of odd reflections is contractible, thereby already demonstrating the effectiveness of using edge-colored graphs to study the exchange property. In contrast, while our approach also employs an edge-colored graph framework, it provides a different type of generalization, carried out within an elementary setting.

%% file: 1.6.tex
I would like to express my heartfelt gratitude to my supervisor, Syu Kato, for his patient and extensive guidance throughout the preparation of the master's thesis, as well as for his helpful suggestions and constructive feedback. The author is also grateful to Istvan Heckenberger for his valuable discussions, comments, and helpful advice, particularly for clarifying proofs in Chapter 2 and 3.2. The author thanks Takuya Saito for suggesting the possibility of \cref{ko} during an informal conversation at a research meeting. Discussions with Yoshiyuki Koga, and Hiroyuki Yamane, to whom I am also grateful, have influenced this work.. The author would like to thank the Kumano Dormitory community at Kyoto University for their generous financial and living assistance

%% file: graph.tex
We first recall basic terminology related to graphs.

\begin{definition}[Edge-Colored Graph]
An \emph{edge-colored graph} is a triple \( (G, \phi, C) \), where:
\begin{itemize}
    \item \( G = (V, E) \) is a graph with a vertex set \( V \) and an edge set \( E \);
    \item \( \phi: E \to C \) is a function that assigns a color \( \phi(e) \in C \) to each edge \( e \in E \), where \( C \) is a set of colors.
\end{itemize}

Additionally, we assume \(\phi\) is surjective.

The graph \( G \) is called \emph{properly colored} if for each vertex the insident edges have distinct colors.
\end{definition}

\begin{definition}

Let \( G \) be a edge-colored graph with a vertex set \( V \) and a color set \( C \).

A \emph{walk} in \( G \) is a finite non-null sequence \( W = v_0 e_1 v_1 \dots e_k v_k \), whose terms are alternately vertices and edges, such that for \( 1 \leq i \leq k \), the ends of \( e_i \) are \( v_{i-1} \) and \( v_i \). Since all the graphs we consider are properly colored, we denote the color of \( e_i \) by \( c_i \) and write \( W = v_0 c_1 v_1 \dots c_k v_k \) without the loss of generality. The \emph{length} of \( W \) is defined as \( k \), the number of edges in the walk. 

A walk \( v_0 c_1 v_1 \dots c_k v_k \) is called a \emph{path} if \( v_0, v_1, \dots, v_k \) are all distinct.

If \( W = v_0 c_1 v_1 \dots c_k v_k \) and \( W' = v_k c_{k+1} v_{k+1} \dots c_l v_l \) are walks, then the walk \( v_k c_k v_{k-1} \dots v_1 c_1 v_0 \), obtained by reversing \( W \), is denoted by \( W^{-1} \), and the walk \( v_0 c_1 v_1 \dots c_k v_k c_{k+1} v_{k+1} \dots c_l v_l \), obtained by concatenating \( W \) and \( W' \) at \( v_k \), is denoted by \( WW' \).

A walk is closed if its origin and terminus are the same.
 A closed walk of length \( 0 \) is called an \emph{empty path}, and of length \( 1 \) is called a \emph{loop}.

 A walk is termed \emph{shortest} if there is no shorter walk between the same pair of vertices. Note that  a shortest walk is a path. In a connected graph, for any pair of vertices, a shortest walk (not necessarily unique) always exists.

 A walk is called a \emph{rainbow} if all the edge colors in its sequence are distinct.
    
\end{definition}

\begin{definition}[Rainbow Boomerang Graph]
A edge-colored graph \( G \) which is properly colored is called a \emph{rainbow boomerang graph} when a walk is shortest if and only if it is rainbow.
\end{definition}

\begin{remark}
By definition, a rainbow boomerang graph is  loopless and multiedge-free. 
\end{remark}

\begin{lemma}\label{lemma:odd_cycle} 
A rainbow boomerang graph does not contain any closed walk of odd length. In other words, it is bipartite.

\begin{proof}
To show that a rainbow boomerang graph does not contain a closed walk of length \( 2k+1 \), let \( W = v_0 c_1 v_1 \dots v_{2k} c_{2k+1} v_0 \). We prove by induction on \( k \).

For \( k = 0 \), the walk is a loop, which is impossible.

Now suppose \( k > 0 \). If there exists a subwalk \( v_0 c_1 v_1 \dots v_{l-1} c_l v_l \) of length \( l \leq k \) that is not shortest, then there exists \( s < l \) and a walk \( v_0 d_1 w_1 \dots w_{s-1} d_s v_l \), and the closed walk \( W \) can be decomposed into two smaller closed walks:
\[
v_0 d_1 w_1 \dots w_{s-1} d_s v_l c_l v_{l-1} \dots v_1 c_1 v_0
\]
and
\[
v_0 d_1 w_1 \dots w_{s-1} d_s v_l c_{l+1} v_{l+1} \dots v_{2k} c_{2k+1} v_0.
\]
The lengths of these walks are \( s + l \) and \( (2k+1 - l + s )\), respectively. One of these walks must have odd length, reducing the problem to a smaller \( k \).

Therefore, we may assume all subwalks of length \( \leq k \) are shortest. 

The subwalks of length \( k \), \( v_0 c_1 v_1 \dots c_k v_k \) and \( v_1 c_2 v_2 \dots c_k v_k c_{k+1} v_{k+1} \), are shortest.  
Since the walk of length \( k+1 \), \( v_0 c_1 v_1 \dots c_k v_k c_{k+1} v_{k+1} \), is not shortest, it cannot be a rainbow walk.

These imply that \( c_1 = c_{k+1} \). By the same reasoning, we also have \( c_{2k+1} = c_{k+1} \).
Thus, we have \( c_1 = c_{2k+1} \), which contradicts the proper coloring. Therefore, we conclude the result.

\end{proof}
\end{lemma}

\begin{example}
In the edge colored graphs below, a walk is rainbow if and only if it is shortest walk. However, these graphs are not properly coloered, so are not rainbow boomerang.

\begin{center}
\begin{tikzpicture}
    \node (B) at (-1.5, 0) {};
    \node (C) at (1.5, 0) {};

    \draw (B) to[bend left=30] node[midway, above] {\( c \)} (C);
    \draw (B) to[bend right=30] node[midway, below] {\( c \)} (C);
 \fill (B) circle (2pt);
    \fill (C) circle (2pt);

\end{tikzpicture}
\hspace{1cm}
\begin{tikzpicture}
    \node (A) at (0, 1.5) {};
    \node (B) at (-1.5, 0) {};
    \node (C) at (1.5, 0) {};

    \draw (A) -- node[left] {\( c \)} (B);
    \draw (B) -- node[below] {\( c \)} (C);
    \draw (C) -- node[right] {\( c \)} (A);

    \fill (A) circle (2pt);
    \fill (B) circle (2pt);
    \fill (C) circle (2pt);
\end{tikzpicture}
\hspace{1cm}
\begin{tikzpicture}
    \foreach \i in {1,2,3,4,5} {
        \node[draw, fill=black, circle, inner sep=1pt] (V\i) at (72*\i:2cm) {};
    }
    
    \foreach \i/\j in {1/2, 2/3, 3/4, 4/5, 5/1, 1/3, 1/4, 2/4, 2/5, 3/5} {
        \draw (V\i) -- (V\j) node[midway, above] {\( c \)};
    }
\end{tikzpicture}
\end{center}

\end{example}

\begin{example}\label{sq}
Consider when the following square graph is a rainbow boomerang graph.  
Suppose that \( c_1, \dots, c_4 \) are distinct. Then, the walk  
\( v_1 c_1 v_2 c_2 v_3 c_3 v_4 \) is a rainbow walk, but since the shorter walk \( v_1 c_4 v_4 \) exists, it is not shortest.  
Thus, the graph is not a rainbow boomerang graph.  

Taking into account the condition of proper coloring, the edge-colored graph below is a rainbow boomerang graph if and only if \( c_1 = c_3 \neq c_2 = c_4 \).

\begin{tikzpicture}
    \node (v1) at (0,0) {\( v_1 \)};
    \node (v2) at (2,0) {\( v_2 \)};
    \node (v3) at (2,2) {\( v_3 \)};
    \node (v4) at (0,2) {\( v_4 \)};
    
    \draw (v1) -- (v2) node[midway, below] {\( c_1 \)};
    \draw (v2) -- (v3) node[midway, right] {\( c_2 \)};
    \draw (v3) -- (v4) node[midway, above] {\( c_3 \)};
    \draw (v4) -- (v1) node[midway, left] {\( c_4 \)};
\end{tikzpicture}
\end{example}

\begin{lemma}\label{unipue_colors}
For any two points \( x, y \) in a rainbow boomerang graph, if a rainbow shortest path 
\( W = x c_1 v_1 c_2 \dots v_{l-1} c_l y \) is chosen, the set of colors \( \{ c_0, c_1, \dots, c_l \} \) does not depend on the choice of the rainbow shortest path \( W \).
\end{lemma}

\begin{proof}
To prove the claim, consider a closed walk of length \( 2l \):
\[
v_0 c_1 v_1 c_2 \dots v_{l-1} c_l v_l d_1 v_{l+1} \dots v_{2l-1} d_l v_0,
\]
and assume that the distance between \( v_0 \) and \( v_l \) is \( l \). We aim to show that \( \{ c_1, \dots, c_l \} = \{ d_1, \dots, d_l \} \) by induction on \( l \).

If \( l = 0, 1 \), the assertion is trivial. If \( l = 2 \), it follows from \cref{sq}.

For \( l > 2 \), suppose \( d_1 \neq c_i \) for all \( i \). Then, the walk \( v_0 c_1 v_1 c_2 \dots c_l v_l d_1 v_{l+1} \) would be rainbow, but since a shorter walk \( v_0 d_l v_{2l-1} \dots v_{l+2} d_2 v_{l+1} \) exists, it is not shortest. Hence, we must have \( d_1 = c_i \) for some \( 1 \leq i < l \). In this case, the walk \( v_{i-1} c_i v_i \dots c_l v_l c_i v_{l+1} \) is not rainbow. Thus, there exists a walk of length less than \((l - i + 2 )\) between \( v_{l+1} \) and \( v_{i-1} \).

In particular, if \( i \neq 1 \), then the closed walk can be decomposed into smaller closed walks, reducing the length. By a similar argument, unless \( d_1 = c_1, d_2 = c_2, \dots, d_l = c_l \), the walk can always be divided into smaller closed walks, completing the proof.
\end{proof}

\begin{definition}
For any two points \( x, y \) in the same connected component of a rainbow boomerang graph, 
define the subset \( C(x, y) \subseteq C \) as the set of colors that appear in a rainbow shortest path between \( x \) and \( y \).
\end{definition}

The following proposition (see also \cite[Propsition 5.3.5]{gorelik2022root}) serves as the motivation for naming these graphs \emph{rainbow boomerang} and is fundamental for establishing odd Verma theorem. This is also often called exchange property.
\begin{proposition}\label{g.rb2.kai}
Let \( G \) be a rainbow boomerang graph. Let \( k \) be a positive integer.  
If there exists a rainbow walk  
\(
v_0 c_0 v_1 c_1 \dots c_k v_{k+1}
\)
and an edge \( v_{k+1} c_0 v_{k+2} \),  
then there exists a rainbow walk  
\(
v_{k+2} d_1 v_{k+3} d_2 \dots v_{2k+1} d_k v_0
\)
such that \( \{ c_1, c_2, \dots, c_k \} = \{ d_1, d_2, \dots, d_k \} \).
\begin{center}
\begin{tikzpicture}
    \node (A1) at (0, 0) {$v_1$};
    \node (A2) at (2, 0) {$v_2$};
    \node (A3) at (4, 0) {$v_3$};
    \node (Ak1) at (6, 0) {$v_{k}$};
    \node (Ak) at (8, 0) {$v_{k+1}$};

    \node (B1) at (0, -2) {$v_0$};
    \node (B2) at (2, -2) {$v_{2k+1}$};
    \node (B3) at (4, -2) {$v_{2k}$};
    \node (Bk1) at (6, -2) {$v_{k+3}$};
    \node (Bk) at (8, -2) {$v_{k+2}$};

    \draw (A1) -- node[above] {$c_1$} (A2);
    \draw (A2) -- node[above] {$c_2$} (A3);
    \draw[dotted] (A3) -- (Ak1); 
    \draw (Ak1) -- node[above] {$c_k$} (Ak);

    \draw[dashed] (B1) -- node[below] {$d_k$} (B2);
    \draw[dashed] (B2) -- node[below] {$d_{k-1}$} (B3);
    \draw[dotted] (B3) -- (Bk1); 
    \draw[dashed] (Bk1) -- node[below] {$d_1$} (Bk);

    \draw (A1) -- node[left] {$c_0$} (B1);
    \draw (Ak) -- node[right] {$c_0$} (Bk);
\end{tikzpicture}
\end{center}
\end{proposition}

\begin{proof}
Consider the walk of length \( k+2 \):  
\(
v_0 c_0 v_{1} c_1 v_2 c_2 \dots c_k v_{k+1} c_0 v_{k+2},
\)
which is not rainbow. Hence, there must exist a shortest walk of length \( k+1 \) or less from \( v_{k+2} \) to \( v_0 \).  
If the length is \( k+1 \), this would result in an odd-length closed walk, which contradicts the fact that \( G \) is bipartite.  
If the length is \( k-1 \) or less, it would contradict the assumption that the walk of length \( k+1 \),  
\(
v_0 c_0 v_1 c_1 \dots c_k v_{k+1},
\)
is shortest.

Consequently, there exists a rainbow walk:  
\(
v_{k+2} d_1 v_{k+3} d_2 \dots v_{2k+1} d_k v_0.
\)

The remaining claims follow from \cref{unipue_colors}.
\end{proof}

\begin{lemma}\label{2.1.kiketu_rb2}
Let 
\(
v_0 c_1 v_1 c_2 v_2 \dots c_k v_k
\)
be a walk in the rainbow boomerang graph \( G \). Then \( C(v_0, v_k) \) is nothing other than the set of colors that appear an odd number of times along this walk.
\end{lemma}
\begin{proof}
We proceed by induction on \( k \).

For \( k = 0, 1 \), the statement is trivial.

Now, assume \( k > 1 \). If \( \{c_1, \dots, c_k\} \) has no repetitions, there is nothing to prove. Thus, assume that \( \{c_1, c_2, \dots, c_k\} \) contains repetitions. Let \( i \) (\(1 < i \leq k\)) be the smallest index such that \( c_i \) repeats in \( \{c_1, c_2, \dots, c_i\} \). Let \( j \) (\(1 \leq j < i\)) be the unique index such that \( c_j = c_i \).

The subwalk
\(
v_{j-1} c_{j} v_j c_{j-1} \dots v_{i-2} c_{i-1} v_{i-1}
\)
is rainbow. Applying \cref{g.rb2.kai} to the subwalk
\(
v_{j-1} c_j v_j \dots v_{i-1} c_i v_i,
\)
we obtain a walk of length \( k-2 \) from \( v_0 \) to \( v_k \) such that the parity of the occurrences of each color is the same as in the original walk.

Thus, the length of the walk decreases, and the claim follows by induction.
\end{proof}

\begin{corollary}\label{g.RB1}
   For any three points \( x, y, z \) in the same connected component of a rainbow boomerang graph, 
\( y = z \) if and only if \( C(x, y) = C(x, z) \).
\end{corollary}

\begin{proof}
If \( y = z \), it is clear that \( C(x, y) = C(x, z) \).

Conversely, if \( C(x, y) = C(x, z) \), then by \cref{2.1.kiketu_rb2}, we must have \( C(y, z) = \emptyset \).
\end{proof}

\begin{corollary} \label{wwwww}
For a connected rainbow boomerang graph \( G \) and a color \( c \) of \( G \), the edge-colored graph obtained by removing all edges of color \( c \) from \( G \) consists of two connected components, each of which is a rainbow boomerang graph.
\end{corollary}

\begin{proof}
This follows immediately from \cref{2.1.kiketu_rb2} and \cref{g.RB1}.
\end{proof}

\begin{definition}\label{d1}
    Let \( G \) be an edge-colored graph with color set \( C \). Let \( D \subseteq C \). We define an equivalence relation \( \sim_D \) on \( V \) as follows: For \( x, y \in V \), we say \( x \sim_D y \) if there exists a walk from \( x \) to \( y \) consisting only of edges with colors in \( D \). We denote the equivalence class of \( x \) by \( [x] \).

    We define the edge-colored graph \( G/D \) as follows:
    \begin{itemize}
        \item The vertex set is \( V/\sim_D \), the set of equivalence classes under \( \sim_D \);
        \item The color set is \( C \setminus D \);
        \item There is an edge of color \( c \in C \setminus D \) between \( [x] \) and \( [y] \) in \( G/D \) if and only if there exist \( u \in [x] \) and \( v \in [y] \) such that there is an edge of color \( c \) between \( u \) and \( v \) in \( G \).
    \end{itemize}
Given a walk \( W \) in \( G \):
\[
v_0 c_0 v_1 c_1 \dotsc c_{k-1} v_k,
\]
we define the \emph{induced walk} \( \overline{W} \) in \( G/D \) as:
\[
[v_0] [c_0] [v_1] [c_1] \dotsc [c_{k-1}] [v_k],
\]
where \( [c_i] = c_i \) if \( [v_i] \neq [v_{i+1}] \), and \( [c_i] \) represents an empty walk if \( [v_i] = [v_{i+1}] \).

\end{definition}

\begin{proposition}\label{3.1.quo}
 Let \( G \) be an rainbow boomerang graph with color set \( C \). Let \( D \subseteq C \).
    Then, the graph \( G/D \) is a rainbow boomerang graph.
\end{proposition}

\begin{proof}
Noting that the colors appearing in \( \bar{W} \) are exactly those among the colors appearing in \( W \) that belong to \( C \setminus D \), it follows from \cref{g.RB1}.
\end{proof}


\begin{example}\label{tree}

    A connected edge-colored tree is a rainbow boomerang graph if and only if all edges have distinct colors.
More generally, in a rainbow boomerang graph, the color of any \emph{bridge} (i.e., an edge whose removal disconnects the graph) is distinct from the colors of all other edges.
\end{example}

\begin{example}[Cycle Graph \( C_n \)] \label{g.cn}
A \emph{cycle graph} \( C_n \) is a graph defined as follows:
\begin{itemize}
    \item \textbf{Vertex set:} \( \{ v_1, v_2, \dots, v_n \} \);
    \item \textbf{Edges:} An edge \( e_i \) connects \( v_i \) and \( v_{i+1} \) for \( 1 \leq i \leq n-1 \), and an edge \( e_n \) connects \( v_n \) and \( v_1 \).
\end{itemize}

Now, consider an edge-colored graph obtained by coloring each edge \( e_i \) of \( C_n \) with a color \( c_i \). This edge-colored graph becomes a \emph{rainbow boomerang graph} if and only if the following conditions hold:
\begin{itemize}
    \item \( n = 2m \) for some integer \( m \neq 1\);
    \item \( c_1, \dots, c_m \) are pairwise distinct;
    \item \( c_1 = c_{m+1}, c_2 = c_{m+2}, \dots, c_m = c_{2m} \).
\end{itemize}

This characterization is a consequence of proof of \cref{g.rb2.kai}.
\end{example}

The following example constitutes important background in this work.

\begin{example}\label{g.cox}
Let \( W \) be a (not necessarily crystallographic) finite Coxeter group with \( S \) as the set of simple reflections. The Cayley graph \( \mathrm{Cay}(W, S) \) has vertices corresponding to the bases of the root system, and its edges are colored by the reflecting hyperplanes. When colored in this way, \( \mathrm{Cay}(W, S) \) becomes a rainbow boomerang graph due to the well-known fact \cref{g.lwnw}.
\end{example}

 \begin{example}
 The (finite) Young lattice (and its type~D variant), as well as suitable quotients thereof, naturally form examples of rainbow boomerang graphs. These examples are described in detail in \cite{Hirota2025}.

In this work, we do not consider the lattice structure; however, this lattice structure has been actively studied in recent years. For example, see \cite{coggins2024visual,stanley2012topics}. Our notion of quotient is consistent with this lattice structure in a certain sense.
\end{example}

\begin{remark}\label{111110}
Rainbow connection is a concept in graph theory that has been actively studied in recent years, with researchers exploring its theoretical properties and practical applications in areas such as secure communication and network design \cite{li2013rainbow}.
Specifically, the following concepts are commonly studied:

\begin{itemize}
    \item An edge-colored graph \( G \) is \emph{(strongly) rainbow connected} if any two vertices are connected by a rainbow path (which is also shortest).
\end{itemize}

Furthermore, the following concept has been studied in \cite{chandran2018algorithms}, and to the best of the author's knowledge, it is the closest to the class of graphs considered in this work.

\begin{itemize}
    \item An edge-colored graph \( G \) is said to be \emph{very strongly rainbow connected} if every shortest path in \( G \) is always a \emph{rainbow path}.
\end{itemize}

The \emph{rainbow boomerang graph}, as considered in this work, is defined as a class of graphs where the converse also holds, imposing much stricter constraints. In fact, for a given graph \( G \), while there always exists a trivial \emph{very strongly rainbow coloring} (where every edge is assigned a distinct color), a coloring that makes \( G \) a rainbow boomerang graph may not exist. This is due to the strong restrictions on the coloring of closed walks, as observed in \cref{g.cn}. This is illustrated in the following examples.
\end{remark}

\begin{example}
    The following bipartite graphs do not admit a rainbow boomerang coloring:

    \begin{center}
    \begin{tikzpicture}[scale=1.05]
        \begin{scope}[xshift=0cm, scale=0.7]
            \node[draw, fill=black, circle, inner sep=2pt] (A) at (-1, -1) {};
            \node[draw, fill=black, circle, inner sep=2pt] (B) at (1, -1) {};
            \node[draw, fill=black, circle, inner sep=2pt] (C) at (1, 1) {};
            \node[draw, fill=black, circle, inner sep=2pt] (D) at (-1, 1) {};
            \node[draw, fill=black, circle, inner sep=2pt] (E) at (0, 0) {};

            \draw (A) -- (B);
            \draw (B) -- (C);
            \draw (C) -- (D);
            \draw (D) -- (A);
            \draw (E) -- (A);
            \draw (E) -- (C);
        \end{scope}

        \begin{scope}[xshift=4cm, scale=0.7]
            \foreach \i in {1,...,6} {
                \node[circle, draw, fill=black, inner sep=2pt] (\i) at ({60 * (\i - 1)}:2cm) {};
            }

            \foreach \i in {1,...,5} {
                \pgfmathtruncatemacro{\j}{\i + 1} 
                \draw (\i) -- (\j);
            }
            \draw (6) -- (1); 

            \draw (1) -- (4); 
            \draw (3) -- (6); 
        \end{scope}
    \end{tikzpicture}
    \end{center}
\end{example}

\begin{example}
    A complete bipartite graph \( K_{m,n} \) (\( m \leq n \)) admits a rainbow boomerang coloring if and only if \( m = 1 \) or \( m = n = 2 \).
\end{example}

\begin{example}\label{g.qn}
The hypercube graph \( Q_n \) is defined as follows:
\begin{itemize}
    \item Vertex set: \( (\mathbb{Z}/2\mathbb{Z})^{\oplus n} \);
    \item Edge set: An edge exists between two vertices if they differ in exactly one coordinate.
\end{itemize}
 The graph \( Q_n \) admits a natural proper edge coloring with the color set of size \( n \), and according to this cloring, \( Q_n \) is a rainbow boomerang graph. 

Let \( G \) be a connected rainbow boomerang graph with a color set \( C \) of size \( n \). Fix a vertex \( x \) in \( G \). By \cref{g.RB1}, the vertice \(y\) of \( G \) can be uniquely characterized by the set \(C(x,y)\).

In this way, the vertex set of \( G \) can be viewed as a subset of \( (\mathbb{Z}/2\mathbb{Z})^{\oplus n} \), and through this identification, the following clearly holds:

\end{example}

\begin{proposition}
    A connected rainbow boomerang graph with a color set of size \( n \) can be embedded into the hypercube graph \( Q_n \) as edge colored graphs.
\end{proposition}

The following theorem, whose possibility was suggested by Takuya Saito, characterizes the concept of a rainbow boomerang graph in terms of its underlying graph.

\begin{theorem}\label{ko}
A connected sub-edge-colored graph \( G \) in the hypercube \( Q_n \) is a rainbow boomerang graph if and only if \( G \cap Q_{n'} \) is connected for any subhypercube \( Q_{n'} \) in \( Q_n \).
\end{theorem}
\begin{proof}
For a walk \( W \) on \( G \subseteq Q_n \), if \( W \) is rainbow, then it is clearly shortest.

Suppose that \( G \subseteq Q_n \) is a rainbow boomerang graph, and take \( x, y \in G \cap Q_{n'} \).  
    There exists a rainbow path between \( x \) and \( y \) in \( G \subseteq Q_n \), but since all rainbow paths between \( x \) and \( y \) in \( Q_n \) can be realized within \( Q_{n'} \), it follows that \( G \cap Q_{n'} \) is connected.

    Conversely, suppose that \( G \subseteq Q_n \) is not a rainbow boomerang graph.  
    Then, there exists a shortest path \( W \) between some \( x, y \) that is not rainbow in \( G \).  
    Since any subwalk of a shortest path is also a shortest path, we can write  
    \[
    W = x c_0 v_1 c_1 \dots v_l c_l v_{l+1} c_0 y,
    \]
    where \( c_0, c_1, \dots, c_l \) are distinct.  
    We proceed by induction on \( l \) to show that there exists a subhypercube \( Q_{n'} \) such that \( G \cap Q_{n'} \) is disconnected.

Base Case (\( l = 1 \)):  
    Consider the subhypercube \( Q_1 \) determined by the tuple \( (x, y, c_1) \).  
    Since the edge \( x c_1 y \) is not contained in \( G \), it follows that \( G \cap Q_1 \) is disconnected.
    
    Inductive Step (\( l > 1 \)):  
    In \( Q_n \), there exists a rainbow path  
    \(
    x c_1 w_2 c_2 \dots w_l c_l y.
    \)
    Consider the subhypercube \( Q_l \) determined by the tuple \( (x, y, c_1, \dots, c_l) \).  
    Any shortest path between \( x \) and \( y \) in \( G \cap Q_l \) has a length of at least \( l + 1 \), and it must not be rainbow.  
    Setting \( n = l \) and appropriately choosing a subwalk, we can apply induction to complete the proof.

\end{proof}

%% file: A.1.tex
\subsection{Basics and examples}

See \cite[Section 9,10]{heckenberger2020hopf} for basic material about Weyl groupoids.

\begin{definition}\cite{heckenberger2020hopf}
    An edge-colored graph \( G \) with vertex set \( V \) is called a \emph{semi Cartan graph} (also known as a \emph{Cartan scheme}) if it is equipped with:
    \begin{itemize}
        \item a non-empty finite set \( I \) of colors,
        \item and a label set \( \{ A^{x} \}_{x \in V} \), where each \( A^{x} \) is a generalized Cartan matrix of size \( \#I \times \#I \) (in the sense of \cite{kac1990infinite}),
    \end{itemize}
    satisfying the following conditions:
    \begin{enumerate}[label=\textbf{(CG\arabic*)}, leftmargin=1.5cm]
        \item \( G \) is properly colored (i.e., edges emanating from the same vertex have distinct colors) and \( \#I \)-regular (i.e., each vertex is incident to exactly \( \#I \) edges).
        \item If two vertices \( x \) and \( y \) are connected by an edge of color \( i \), then the \( i \)-th row of \( A^{x} \) equals the \( i \)-th row of \( A^{y} \).
    \end{enumerate}
    The underlying edge-colored graph of a semi Cartan graph \( G \) is called the \textit{exchange graph} and is denoted by \( E(G) \). When illustrating  \( G \), we omit loops for simplicity, thanks to \textbf{(CG1)}.
    
    The size of \( I \) is called the \emph{rank} of \( G \).
\end{definition}

For \( x \in V \), define \( r_i x \in V \) as the vertex connected to \( x \) by an edge of color \( i \).  Then, \( r_i \) is an involution on \( V \).

For each \( x \in V \), consider a copy \( (\mathbb{Z}^{I})^{x} \) of \( \mathbb{Z}^{I} \) associated with \( x \). The standard basis of \( (\mathbb{Z}^{I})^{x} \) is denoted by \( \{ \alpha_i^{x} \}_{i \in I} \).

The standard basis of \( \mathbb{Z}^{I} \) is also denoted by \( \{ \alpha_i \}_{i \in I} \). We define a standard isomorphism \( \varphi^{x}: \mathbb{Z}^{I} \to \mathbb{Z}^{I} \) for each \( x \), which maps \( \alpha_i^{x} \) to \( \alpha_i \) for \( i \in I \).

For each \( i \in I \) and \( x \in V \), define \( s_i^x \in \operatorname{Hom}_{\mathbb{Z}}((\mathbb{Z}^{I})^x, (\mathbb{Z}^{I})^{r_i x}) \) by the mapping:
\[
    \alpha_j^x \mapsto \alpha_j^{r_i x} - a_{ij}^x \alpha_i^{r_i x}, \quad \text{for } j \in I.
\]

When the context is clear, the subscript \( x \) in \( s_i^x \) may be omitted. Additionally, it is sometimes expressed as a composition with the identity map \( \operatorname{id}_x \) at a vertex \( x \) to emphasize the starting or ending points of the mapping.

\begin{remark}
Our \textbf{(CG1)} is equivalent to \textbf{(CG1)} in \cite{heckenberger2020hopf}.
\end{remark}

\begin{definition}[Semi Weyl Groupoid]
The \emph{semi Weyl groupoid} \( W(G) \) of \( G \) is the category with objects \( V \), where the morphisms from \( x \) to \( r_{i_t} \cdots r_{i_1} x \) are elements of \( \operatorname{Hom}_{\mathbb{Z}}((\mathbb{Z}^{I})^{x}, (\mathbb{Z}^{I})^{r_{i_t} \cdots r_{i_1} x}) \) of the form 
\[
s_{i_t}^{r_{i_{t-1}} \cdots r_{i_1} x} \cdots s_{i_2}^{r_{i_1} x} s_{i_1}^{x}.
\]
We denote the set of such morphisms as \( \operatorname{Hom}_{W(G)}(x, r_{i_t} \cdots r_{i_1} x) \). The composition of morphisms is defined by the natural composition of these maps.
\end{definition}

By the above construction, the semi Weyl groupoid indeed becomes a groupoid due to \textbf{(CG2)}. For a general connected groupoid \( W \), note that the group structure of \( \operatorname{Aut}_W(x) = \operatorname{Hom}_W(x, x) \) does not depend on the choice of \( x \). An element of \( \operatorname{Hom}_{W(G)}(x, y) \) can be regarded as an element of \( \operatorname{Aut}_{\mathbb{Z}}(\mathbb{Z}^{I}) \) via \( \varphi^x \) and \( \varphi^y \).

\begin{definition}[Real Roots] \label{R^q}\cite{heckenberger2020hopf}
For each \( x \in V \), define the set of \emph{real roots} \( R^{x} \) as subsets of \( (\mathbb{Z}^{I})^{x} \) of the form:
\[
R^{x} := \{ w \alpha_i^{y} \mid w \in \operatorname{Hom}_{W(G)}(y, x), \, y \in V, \, i \in I \}.
\]
Let the set of \emph{positive real roots} be defined as:
\[
R^{x+} := \bigl(R^{x} \cap \bigl(\sum_{i \in I} \mathbb{Z}_{\geq 0} \alpha_i^{x} \bigr)\bigr).
\]
\end{definition}

A semi Cartan graph is said to be \emph{finite} when \( \#R^{x} < \infty \).

\begin{definition}\cite{heckenberger2020hopf}
    A semi Cartan graph \( G \) is called a \emph{Cartan graph} if it satisfies the following conditions:
    \begin{enumerate}[label=\textbf{(CG\arabic*)}, align=left, leftmargin=*]
        \item[\textbf{(CG3)}] For all \( x \in V \), \( R^{x} = R^{x+} \cup (-R^{x+}) \).
        \item[\textbf{(CG4)}] If \( w \in \operatorname{Hom}_W(G)(x, y) \) and \( w \alpha_i^{y} \in R^{x+} \) for all \( i \in I \), then \( w = \operatorname{id}_{x} \). In particular, we have \( x = y \).
    \end{enumerate}

    A semi Weyl groupoid arising from a (finite) Cartan graph is called a \emph{(finite) Weyl groupoid}.
\end{definition}

\begin{remark}
    Our \textbf{(CG4)} is equivalent to \textbf{(CG4)} of \cite[Remark 1.6]{heckenberger2020hopf} by \cite{azam2015classification} , [\cite{heckenberger2020hopf}, Corollary 9.3.8] and \cref{lem:root_bij}.
\end{remark}

\begin{remark}
In existing literature, such as \cite{heckenberger2020hopf}, groupoids arising from semi Cartan graphs are also referred to as Weyl groupoids. On the other hand, there is a convention of using the term Weyl groupoid where generalized root system would be more appropriate. Indeed, as in the case of classical BC types, groupoids associated with distinct Cartan graphs can be isomorphic. While adhering to this convention, we distinguish groupoids associated with semi Cartan graphs, which are not Cartan graphs, by calling them semi Weyl groupoids to avoid confusion.
\end{remark}

\begin{lemma}[\cite{heckenberger2020hopf}, Lemma 9.1.19]\label{lem:root_bij}
    Let \( G \) be a semi Cartan graph satisfying (CG3). Then \( s_i^{x} \) provides a bijection between the sets 
    \[
    \big(R^{x} \setminus \{ -\alpha_i^{x} \}\big) \quad \text{and} \quad \big(R^{r_i x} \setminus \{ -\alpha_i^{r_i x} \}\big).
    \]
\end{lemma}
The following is a generalization of  \cref{g.lwnw}

\begin{theorem}[\cite{heckenberger2020hopf} Theorem 9.3.5]\label{poid_lwnw}
Let \(G\) be a Cartan graph and \( w \in \operatorname{Hom}_{W(G)}(x, y) \). Define 
\[
l(w) := \min \{ n \mid \operatorname{id}_x s_{i_n} \dots s_{i_1} = w \}
\]
and
\[
N(w) := \# \{ \alpha \in R^{y+} \mid w \alpha \in -R^{x+} \}.
\]
Then, \( l(w) = N(w) \).
\end{theorem}

\begin{remark}\label{w.st}\cite{heckenberger2020hopf}
    A semi Cartan graph is called \textit{standard} if \( A^{x} \) is independent of \( x \in V \).
    
    For a standard Cartan graph \( G \):
    \[
    G \text{ is finite} \iff A^{x} \text{ is of finite type}.
    \]
    This result and the term \textit{"real root"} are from Kac \cite{kac1990infinite} and are consistent with the definitions provided therein.

    In particular, the Weyl groupoid arising from a finite Cartan graph with a single vertex can be identified with the Weyl group of type \( A^x \).
\end{remark}

\begin{example}\cite{heckenberger2020hopf,cuntz2009weyltwo,cuntz2014frieze}\label{w.5}
\begin{center}
    \begin{tikzpicture}[scale=1.33] 

    \node (1) at (0, 0) {\(\begin{bmatrix} 2 & -2\\ -2& 2 \end{bmatrix}\)};
    \node (2) at (2, 0) {\(\begin{bmatrix} 2 & -2\\ -1 & 2 \end{bmatrix}\)};
    \node (3) at (4, 0) {\(\begin{bmatrix} 2 & -3 \\ -1 & 2 \end{bmatrix}\)};
    \node (4) at (6, 0) {\(\begin{bmatrix} 2 & -3 \\ -1& 2 \end{bmatrix}\)};
    \node (5) at (8, 0) {\(\begin{bmatrix} 2 & -2 \\ -1& 2 \end{bmatrix}\)};
    \node (6) at (8, -1.5) {\(\begin{bmatrix} 2 & -2\\ -2& 2 \end{bmatrix}\)};
    \node (7) at (6, -1.5) {\(\begin{bmatrix} 2 & -1\\ -2& 2 \end{bmatrix}\)};
    \node (8) at (4, -1.5) {\(\begin{bmatrix} 2 & -1 \\ -3 & 2 \end{bmatrix}\)};
    \node (9) at (2, -1.5) {\(\begin{bmatrix} 2 & -1\\ -3& 2 \end{bmatrix}\)};
    \node (10) at (0, -1.5) {\(\begin{bmatrix} 2 & -1\\ -2& 2 \end{bmatrix}\)};

 \draw (1) -- (2) node[midway, above] {1};
\draw (2) -- (3) node[midway, above] {2};
\draw (3) -- (4) node[midway, above] {1};
\draw (4) -- (5) node[midway, above] {2};
\draw (5) -- (6) node[midway, right] {1};
\draw (6) -- (7) node[midway, above] {2};
\draw (7) -- (8) node[midway, above] {1};
\draw (8) -- (9) node[midway, above] {2};
\draw (9) -- (10) node[midway, above] {1};
\draw (10) -- (1) node[midway, left] {2};

\end{tikzpicture}
\end{center}

From the semi Cartan graph of rank two above, considering (CG2), the sequence \( (2,1,3,1,2,2,1,3,1,2) \) naturally corresponds to it. Determining the real root system of this semi Cartan graph can be confirmed to be equivalent to considering a frieze with this sequence as the quiddity sequence. In this case, the frieze is as follows, confirming that it is a finite Cartan graph.

For example, when the top-left vertex of the graph above is denoted as \( x \), the set 
\[
R^{x+} = \{\alpha_1^x, 2\alpha_1^x + \alpha_2^x, \alpha_1^x + \alpha_2^x, \alpha_1^x + 2\alpha_2^x, \alpha_2^x\},
\]
corresponds to the bold column in the following frieze. Similarly, it can be confirmed that the real root system of the adjacent vertex corresponds to the sequence shifted by one position.
Furthermore, the frieze extended to negative entries can also be interpreted in terms of negative roots.

\begin{center}
\begin{tikzpicture}[xscale=1, yscale=1.2, every node/.style={anchor=mid}]

\foreach \x/\val in {0/0, 1/\textbf{0}, 2/0, 3/0, 4/0, 5/0, 6/0, 7/0, 8/0, 9/0, 10/0, 11/0, 12/0, 13/0, 14/0} {
    \node at (\x, 0) {\val};
}

\foreach \x/\val in {0/\textbf{1}, 1/\textbf{1}, 2/1, 3/1, 4/1, 5/1, 6/1, 7/1, 8/1, 9/1, 10/1, 11/1, 12/1, 13/1} {
    \node at (\x+0.5, -1) {\val};
}

\foreach \x/\val in {0/2, 1/\textbf{2}, 2/\textbf{1}, 3/3, 4/1, 5/2, 6/2, 7/1, 8/3, 9/1, 10/2, 11/2, 12/1, 13/3, 14/1} {
    \node at (\x, -2) {\val};
}

\foreach \x/\val in {0/3, 1/\textbf{1}, 2/\textbf{2}, 3/2, 4/1, 5/3, 6/1, 7/2, 8/2, 9/1, 10/3, 11/1, 12/2, 13/2} {
    \node at (\x+0.5, -3) {\val};
}

\foreach \x/\val in {0/1, 1/1, 2/\textbf{1}, 3/\textbf{1}, 4/1, 5/1, 6/1, 7/1, 8/1, 9/1, 10/1, 11/1, 12/1, 13/1, 14/1} {
    \node at (\x, -4) {\val};
}

\foreach \x/\val in {0/0, 1/0, 2/\textbf{0}, 3/0, 4/0, 5/0, 6/0, 7/0, 8/0, 9/0, 10/0, 11/0, 12/0, 13/0} {
    \node at (\x+0.5, -5) {\val};
}

\end{tikzpicture}
\end{center}

By a similar argument, it can be seen that a connected (simply connected) finite Cartan graphs of rank two is equivalent to the concept of frieze patterns. In particular, according to the classification results of Conway and Coxeter \cite{conway1973triangulated}, the isomorphism classes are parametrized by the triangulations of regular polygons. In particular, the current example corresponds to a triangulation of regular pentagon.

\end{example}

\begin{definition}

A morphism of vertex-labeled edge-colored graphs is a graph morphism that preserves both the labels of the vertices and the colors of the edges.

Below, let the semi Cartan graph be connected.  
Consider a vertex-labeled edge-colored graph morphism \( \widetilde{G} \to G \) between semi Cartan graphs with the same color set \( I \). We call \((\widetilde{G}, G, \pi)\) a covering.

\end{definition}

\begin{proposition}[\cite{heckenberger2020hopf}, Proposition 10.1.5]
Let \( (\widetilde{G}, G, \pi) \) be a covering. Then there exists a natural functor on the semi-Weyl groupoid:
\[
F_{\pi}: W(\widetilde{G}) \to W(G),
\]
which induces an injective homomorphism
\[
\operatorname{Aut}_{W(\widetilde{G})}(y) \rightarrow   \operatorname{Aut}_{W(G)}(\pi(y))
\]
for each vertex \( y \in \widetilde{G} \).
\end{proposition}

\begin{definition}\cite{heckenberger2020hopf}
A semi-Cartan graph \( G \) is called \textbf{simply connected} if the map \( \pi \) is an isomorphism for every covering \( (\widetilde{G}, G, \pi) \).

Equivalently, \( G \) is simply connected if
\[
\# \operatorname{Hom}_{W(G)}(x, y) \leq 1 \quad \text{for all } x, y \in V.
\]
\end{definition}

\begin{proposition}[\cite{heckenberger2020hopf}, Proposition 10.1.6]\label{prop:simconnec}
Let \( G \) be a Cartan graph. For \( x \in V(G) \) and a subgroup \( U \subseteq \operatorname{Aut}_{W(G)}(x) \), there exists a covering \( (\widetilde{G}, G, \pi) \) and a vertex \( \widetilde{x} \in V(\widetilde{G}) \) such that:
\[
\pi(\widetilde{x}) = x \quad \text{and} \quad F_{\pi}(\operatorname{Aut}_{W(\widetilde{G})}(\widetilde{x})) = U.
\]
Moreover, such a covering is unique up to isomorphism, and
\[
\# \pi^{-1}(x) = [\operatorname{Aut}_{W(G)}(x) : U].
\]

In particular, a simply connected covering \( \operatorname{SC}(G) \) of \( G \), as a Cartan graph, always exists and is unique up to isomorphism.
\end{proposition}

\begin{example}
By (CG4), the vertex set \( V \) of a connected simply connected Cartan graph can be identified with a set \( \{ w \operatorname{id}_x \mid w \in \operatorname{Hom}_{W(G)}(x, y), \, y \in V \} \), where \( x \in V \) is fixed. Clearly, a connected Cartan graph is loopless if and only if it is simply connected. If \( G \) is standard, then \( SC(G) \), as a graph, is the same as the Cayley graph of the Weyl group. By \cite{inoue2023hamiltonian,yamane2021hamilton}, a simply connected Cartan graph is Hamiltonian (i.e. there exist a path that visits every vertex of a graph exactly once and returns to the starting vertex). 
\end{example}

\begin{example}

The isomorphism classes of connected standard Cartan graphs of type \( A_2 \) correspond to the conjuate classes of subgroups of \( S_3 \) via the following Galois correspondence:

\begin{center}
\begin{tikzpicture}

\node (A1) at (0, 3) {\( S_3 \)};
\node (B1) at (1.5, 2) {\( \mathbb{Z}/3\mathbb{Z} \)};
\node (C1) at (-1.5, 1) {\( \mathbb{Z}/2\mathbb{Z} \)};
\node (D1) at (0, 0) {e};

\draw (A1) -- (B1);
\draw (A1) -- (C1);
\draw (B1) -- (D1);
\draw (C1) -- (D1);

\node (A2) at (5, 3) {\begin{tikzpicture}

\node[draw, circle, fill=black, inner sep=1.5pt] (A) at (0, 0) {};

\end{tikzpicture}};
\node (B2) at (6.5, 2) {\begin{tikzpicture}

\node[draw, circle, fill=black, inner sep=1.5pt] (A) at (0, 0) {};
\node[draw, circle, fill=black, inner sep=1.5pt] (B) at (0.5, 0) {};

\draw[bend right=30] (A) to (B);
\draw[bend left=30] (A) to (B);

\end{tikzpicture}};
\node (C2) at (3.5, 1) {\begin{tikzpicture}

\node[draw, circle, fill=black, inner sep=1.5pt] (A) at (0, 0) {};
\node[draw, circle, fill=black, inner sep=1.5pt] (B) at (0.5,0) {};
\node[draw, circle, fill=black, inner sep=1.5pt] (C) at (1, 0) {};

\draw (A) -- (B);
\draw (B) -- (C);

\end{tikzpicture}};
\node (D2) at (5, -0) {\begin{tikzpicture}

\node[draw, circle, fill=black, inner sep=1.5pt] (A) at (0, 0.5) {};
\node[draw, circle, fill=black, inner sep=1.5pt] (B) at (0.5, 0.5) {};
\node[draw, circle, fill=black, inner sep=1.5pt] (C) at (1, 0.5) {};
\node[draw, circle, fill=black, inner sep=1.5pt] (D) at (1, 0) {};
\node[draw, circle, fill=black, inner sep=1.5pt] (E) at (0.5, 0) {};
\node[draw, circle, fill=black, inner sep=1.5pt] (F) at (0, 0) {};

\draw (A) -- (B);
\draw (B) -- (C);
\draw (C) -- (D);
\draw (D) -- (E);
\draw (E) -- (F);
\draw (F) -- (A);

\end{tikzpicture}};

\draw (A2) -- (B2);
\draw (A2) -- (C2);
\draw (B2) -- (D2);
\draw (C2) -- (D2);

\end{tikzpicture}
\end{center}

In more detail, the graph:
\[
\begin{tikzpicture}

\node[draw, circle, fill=black, inner sep=1.5pt] (A) at (0, 0) {};
\node[draw, circle, fill=black, inner sep=1.5pt] (B) at (0.5, 0) {};
\node[draw, circle, fill=black, inner sep=1.5pt] (C) at (1, 0) {};

\draw (A) -- (B);
\draw (B) -- (C);

\end{tikzpicture}
\]
is represented as:
\[
\begin{tikzpicture}[scale=1.5]

\node at (0, 0) {\(\begin{bmatrix} 2 & -1 \\ -1 & 2 \end{bmatrix}\)};
\node at (2.25, 0) {\(\begin{bmatrix} 2 & -1 \\ -1 & 2 \end{bmatrix}\)};
\node at (4.5, 0) {\(\begin{bmatrix} 2 & -1 \\ -1 & 2 \end{bmatrix}\)};

\draw (0.75, 0) -- node[above] {2} (1.5, 0);
\draw (3, 0) -- node[above] {1} (3.75, 0);

\end{tikzpicture}
\]
which is the Cartan graph of \( \mathfrak{gl}(2|1) \) in the sense of \cref{BNmain}. The corresponding Weyl group is isomorphic to \( \mathbb{Z}/2\mathbb{Z} \).

Additionally, the graph:
\[
\begin{tikzpicture}

\node[draw, circle, fill=black, inner sep=1.5pt] (A) at (0, 0) {};

\end{tikzpicture}
\]
is represented as:
\[
\begin{bmatrix} 2 & -1 \\ -1 & 2 \end{bmatrix}
\]
which is the Cartan graph of \( \mathfrak{sl}_3 \) in the sense of \cref{BNmain}. The corresponding Weyl group is isomorphic to \( S_3 \).
\end{example}

%% file: A.2.tex
\subsection{Exchange property of path subgroupoids}

Below, let the semi Cartan graph be connected.

\begin{definition}\label{A2delta}
The path subgroupoid \( P(G) \) of a semi Cartan graph \( G \) is defined as the subgroupoid of the semi Weyl groupoid \( W(G) \) generated by morphisms of the form:
\[
\left\{ s_{i_t} \dotsb s_{i_1} \operatorname{id}_x \,\middle|\, r_{i_{s+1}} \dotsb r_{i_1} x \neq r_{i_s} \dotsb r_{i_1} x \text{ for } 1 \leq s \leq t - 1 \right\},
\]
where \( x \in V \). For \( x, y \in V \), the set of morphisms between \( x \) and \( y \) in this subgroupoid is denoted by \( \operatorname{Hom}_{P(G)}(x, y) \).

For \( \alpha \in R^{x} \), we define:
\[
\operatorname{orb}(\alpha) := \left\{ w\alpha \,\middle|\, w \in \operatorname{Hom}_{P(G)}(x, y) \right\}\subseteq \bigsqcup_{y \in V} R^{y} ,
\]
and
\[
\Delta := \left\{ \operatorname{orb}(\alpha) \,\middle|\, \alpha \in R^{x} \right\} .
\]
This definition does not depend on the choice of \( x \).

A semi Cartan graph \( G \) is said to be \emph{path simply connected} if \[
\# \operatorname{Hom}_{P(G)}(x, y) = 1 \quad \text{for any } x, y.
\]

Moreover, if \( G \) satisfies \textbf{(CG3)}, this condition is equivalent to the following:
For a fixed point \( x \) and any \( O \in \Delta \), \( \#(O \cap R^{x}) = 1 \) holds.

Furthermore, if \( G \) is finite, this condition is also equivalent to \( \# \Delta = \# R^{x} \).

\end{definition}

\begin{lemma}
 path simply connected semi Cartan  graph is multiedge free
\end{lemma}
\begin{proof}
    If there were two edges with the labels \( i \) and \( j \) between two nodes \( x \) and \( y \), then we would have:
 \[
s_{j} s_{i} \cdot \alpha_{i}^x = s_{j} (-\alpha_{i}^y) = -\alpha_{i}^x - a_{ij}^x \alpha_{j}^x 
\neq \alpha_{i}^x = s_{i} s_{i} \cdot \alpha_{i}^x.
\]
Thus, we have : \(\# \operatorname{Hom}_{P(G)}(x, y) > 1\).
\end{proof}

\begin{example}
The following finite Cartan graph is multiedge-free but not path-simply connected.
\\
\begin{tikzpicture}[scale=3, every node/.style={align=center}] 

\node (1) at (0, 0.7) {$\begin{bmatrix} 2 & -1 \\ -2& 2 \end{bmatrix}$}; 
\node (2) at (1, 0.7) {$\begin{bmatrix} 2 & -1 \\ -2& 2 \end{bmatrix}$}; 
\node (3) at (1, 0) {$\begin{bmatrix} 2 & -1 \\ -2& 2 \end{bmatrix}$}; 
\node (4) at (0, 0) {$\begin{bmatrix} 2 & -1 \\ -2& 2 \end{bmatrix}$}; 

\draw (1) -- (2) node[midway, above] {1}; 
\draw (2) -- (3) node[midway, right] {2}; 
\draw (3) -- (4) node[midway, below] {1}; 
\draw (4) -- (1) node[midway, left] {2};  

\end{tikzpicture}

\end{example}

\begin{example}
    The path subgroupoid of a simply connected Weyl groupoid is the Weyl groupoid itself. Hence, by the definition of simply connectedness, it is path simply connected.
\end{example}
\begin{example}
      semi Cartan trees are trivially path simply connected.
\end{example}

\begin{definition}
When \( G \) is path simply connected, for \( O \in \Delta \), let \( O_{x} \in R^{x} \) be the unique element in \( O \cap R^{x} \). 
Define \( \Delta^{x+} \) as
\[
\Delta^{x+} = \{ O \in \Delta \mid O_x \in R^{x+} \},
\]
and \( \Delta^{\text{pure}+} \) as
\[
\Delta^{\text{pure}+} = \bigcap_{x \in V} \Delta^{x+}.
\]

For instance, if \( G \) is simply connected, then \( \Delta^{\text{pure}+} = \emptyset \).

\end{definition}

Since the path subgroupoid is constructible by its definition, it is effective—just as noted in~\cite{gorelik2022root}—to consider the corresponding edge-colored graph, as we do below.

\begin{definition}\label{RBG}

For a path simply connected Cartan graph \( G \), we define the edge-colored graph \( RB(G) \) as follows:

\begin{itemize}
    \item \textbf{Underlying graph}: The underlying graph of \( G \), with loops removed.

    \item \textbf{Color set \( C \)}: For a fixed \( x \in V \), 
    \[
    C = \Delta^{x+} - \Delta^{\text{pure+}}
    \]

    \item \textbf{Coloring}: Replace each edge between \( z \) and \( y \) colored \( i \) with an edge colored by a unique \( {O} \in C \) such that \( {O}_{z} \in \{\pm \alpha_{i}^{z}\} \) (see \cref{lem:root_bij}).
\end{itemize}
    
\end{definition}

\begin{theorem}\label{A2main}
    \( RB(G) \) of a path simply connected Cartan graph \( G \) is a rainbow boomerang graph.
\end{theorem}

\begin{proof}
In the simply connected case, this follows immediately from \cref{poid_lwnw}.

In general, if \( G \) is connected and path-simply connected, then under the natural identification of \( SC(G) \) with the root system of \( G \), the edge-colored graph obtained from \( RB(SC(G)) \) by removing edges with colors belonging to \( \Delta^{\text{pure+}} \) is a disjoint union of copies of \( RB(G) \), with the number of components equal to the order of the group of automorphisms of an object of \( W(G) \) by \cref{prop:simconnec}. Consequently, \( RB(G) \) is a  rainbow booerang graph by \cref{wwwww}.

\end{proof}

\begin{example}
The symmetric group on four elements can be viewed as the Weyl groupoid of \( \mathfrak{gl}(2|2) \), formed by combining both even and odd reflections. By removing the edges corresponding to even reflections from the Cayley graph of the symmetric group, we obtain a disjoint union of four finite Young lattices \( L(2,2) \). This is consistent with the fact that the Weyl group of \( \mathfrak{gl}(2|2) \) is isomorphic to \( \mathbb{Z}/2\mathbb{Z} \times \mathbb{Z}/2\mathbb{Z} \).

\begin{minipage}{0.45\textwidth}
    \centering
 \begin{tikzpicture}[scale=0.7]
    \node[draw, circle, fill=black, inner sep=1.5pt] (A0) at (0,0.47) {};
    \node[draw, circle, fill=black, inner sep=1.5pt] (B0) at (0.47,0) {};
    \node[draw, circle, fill=black, inner sep=1.5pt] (C0) at (0,-0.47) {};
    \node[draw, circle, fill=black, inner sep=1.5pt] (D0) at (-0.47,0) {};
    \draw (A0) -- (B0) -- (C0) -- (D0) -- (A0);

    \node[draw, circle, fill=black, inner sep=1.5pt] (A1) at (0,1.97) {};
    \node[draw, circle, fill=black, inner sep=1.5pt] (B1) at (0.47,1.5) {};
    \node[draw, circle, fill=black, inner sep=1.5pt] (C1) at (0,1.03) {};
    \node[draw, circle, fill=black, inner sep=1.5pt] (D1) at (-0.47,1.5) {};
    \draw (A1) -- (B1) -- (C1) -- (D1) -- (A1);

    \node[draw, circle, fill=black, inner sep=1.5pt] (A2) at (0,-1.03) {};
    \node[draw, circle, fill=black, inner sep=1.5pt] (B2) at (0.47,-1.5) {};
    \node[draw, circle, fill=black, inner sep=1.5pt] (C2) at (0,-1.97) {};
    \node[draw, circle, fill=black, inner sep=1.5pt] (D2) at (-0.47,-1.5) {};
    \draw (A2) -- (B2) -- (C2) -- (D2) -- (A2);

    \node[draw, circle, fill=black, inner sep=1.5pt] (A3) at (1.5,0.47) {};
    \node[draw, circle, fill=black, inner sep=1.5pt] (B3) at (1.97,0) {};
    \node[draw, circle, fill=black, inner sep=1.5pt] (C3) at (1.5,-0.47) {};
    \node[draw, circle, fill=black, inner sep=1.5pt] (D3) at (1.03,0) {};
    \draw (A3) -- (B3) -- (C3) -- (D3) -- (A3);

    \node[draw, circle, fill=black, inner sep=1.5pt] (A4) at (-1.5,0.47) {};
    \node[draw, circle, fill=black, inner sep=1.5pt] (B4) at (-1.03,0) {};
    \node[draw, circle, fill=black, inner sep=1.5pt] (C4) at (-1.5,-0.47) {};
    \node[draw, circle, fill=black, inner sep=1.5pt] (D4) at (-1.97,0) {};
    \draw (A4) -- (B4) -- (C4) -- (D4) -- (A4);

    \draw (A0) -- (C1);
    \draw (B0) -- (D3);
    \draw (C0) -- (A2);
    \draw (D0) -- (B4);

    \node[draw, circle, fill=black, inner sep=1.5pt] (NA1) at (0,2.5) {}; 
    \node[draw, circle, fill=black, inner sep=1.5pt] (NC2) at (0,-2.5) {}; 
    \node[draw, circle, fill=black, inner sep=1.5pt] (NB3) at (2.5,0) {}; 
    \node[draw, circle, fill=black, inner sep=1.5pt] (ND4) at (-2.5,0) {}; 

    \draw (A1) -- (NA1);
    \draw (C2) -- (NC2);
    \draw (B3) -- (NB3);
    \draw (D4) -- (ND4);

    \draw (NA1) -- (NB3);
    \draw (NB3) -- (NC2);
    \draw (NC2) -- (ND4);
    \draw (ND4) -- (NA1);

    \draw (B1) -- (A3);
    \draw (C3) -- (B2);
    \draw (D2) -- (C4);
    \draw (A4) -- (D1);

\end{tikzpicture}
\end{minipage}
\hfill
\begin{minipage}{0.45\textwidth}
    \centering
\begin{tikzpicture}[scale=0.7]
    \node[draw, circle, fill=black, inner sep=1.5pt] (A0) at (0,0.47) {};
    \node[draw, circle, fill=black, inner sep=1.5pt] (B0) at (0.47,0) {};
    \node[draw, circle, fill=black, inner sep=1.5pt] (C0) at (0,-0.47) {};
    \node[draw, circle, fill=black, inner sep=1.5pt] (D0) at (-0.47,0) {};

    \node[draw, circle, fill=black, inner sep=1.5pt] (A1) at (0,1.97) {};
    \node[draw, circle, fill=black, inner sep=1.5pt] (B1) at (0.47,1.5) {};
    \node[draw, circle, fill=black, inner sep=1.5pt] (C1) at (0,1.03) {};
    \node[draw, circle, fill=black, inner sep=1.5pt] (D1) at (-0.47,1.5) {};
    \draw (A1) -- (B1) -- (C1) -- (D1) -- (A1);

    \node[draw, circle, fill=black, inner sep=1.5pt] (A2) at (0,-1.03) {};
    \node[draw, circle, fill=black, inner sep=1.5pt] (B2) at (0.47,-1.5) {};
    \node[draw, circle, fill=black, inner sep=1.5pt] (C2) at (0,-1.97) {};
    \node[draw, circle, fill=black, inner sep=1.5pt] (D2) at (-0.47,-1.5) {};
    \draw (A2) -- (B2) -- (C2) -- (D2) -- (A2);

    \node[draw, circle, fill=black, inner sep=1.5pt] (A3) at (1.5,0.47) {};
    \node[draw, circle, fill=black, inner sep=1.5pt] (B3) at (1.97,0) {};
    \node[draw, circle, fill=black, inner sep=1.5pt] (C3) at (1.5,-0.47) {};
    \node[draw, circle, fill=black, inner sep=1.5pt] (D3) at (1.03,0) {};
    \draw (A3) -- (B3) -- (C3) -- (D3) -- (A3);

    \node[draw, circle, fill=black, inner sep=1.5pt] (A4) at (-1.5,0.47) {};
    \node[draw, circle, fill=black, inner sep=1.5pt] (B4) at (-1.03,0) {};
    \node[draw, circle, fill=black, inner sep=1.5pt] (C4) at (-1.5,-0.47) {};
    \node[draw, circle, fill=black, inner sep=1.5pt] (D4) at (-1.97,0) {};
    \draw (A4) -- (B4) -- (C4) -- (D4) -- (A4);

    \draw (A0) -- (C1);
    \draw (B0) -- (D3);
    \draw (C0) -- (A2);
    \draw (D0) -- (B4);

    \node[draw, circle, fill=black, inner sep=1.5pt] (NA1) at (0,2.5) {}; 
    \node[draw, circle, fill=black, inner sep=1.5pt] (NC2) at (0,-2.5) {}; 
    \node[draw, circle, fill=black, inner sep=1.5pt] (NB3) at (2.5,0) {}; 
    \node[draw, circle, fill=black, inner sep=1.5pt] (ND4) at (-2.5,0) {}; 

    \draw (A1) -- (NA1);
    \draw (C2) -- (NC2);
    \draw (B3) -- (NB3);
    \draw (D4) -- (ND4);

\end{tikzpicture}
\end{minipage}
\end{example}

\begin{example}
Let \( G \) be a finite Cartan graph of rank 2. Then, \( G \) is multiedge-free (if \( \#V \neq 2 \), this is the case) if and only if \( G \) is path simply connected. In this case, \( RB(G) \) is one of the following:
\begin{itemize}
    \item a line segment ;
    \item a cycle graph \( C_{2n} \) of length \( 2n \) (\( n > 0 \)) .
\end{itemize}
\end{example}

%% file: A.3.tex
Throughout the following, let \( \mathfrak{g} \) denote a regular symmetrizable Kac–Moody Lie superalgebra \cite{bonfert2024weyl,serganova2011kac}. In order to maintain consistency with the formalism of \cite{bonfert2024weyl}, we actively adopt the terminology of Borel subalgebras.

We denote the even and odd parts of \( \mathfrak{g} \) as \( \mathfrak{g}_{\overline{0}} \) and \( \mathfrak{g}_{\overline{1}} \), respectively.

\begin{definition}[\cite{serganova2017representations}]\label{3.1delta}
A Cartan subalgebra of the Lie algebra \( \mathfrak{g}_{\overline{0}} \) is denoted by \( \mathfrak{h} \).

The root space \( \mathfrak{g}_\alpha \) associated with \( \alpha \in \mathfrak{h}^* \) is defined as 
\(
\mathfrak{g}_\alpha := \{ x \in \mathfrak{g} \mid [h, x] = \alpha(h)x \, \text{for all } h \in \mathfrak{h} \}.
\)

The set of roots \( \Delta \) is defined as 
\(
\Delta := \{ \alpha \in \mathfrak{h}^* \mid \mathfrak{g}_\alpha \neq 0 \} \setminus \{0\}.
\)
Each \( \mathfrak{g}_\alpha \) is either purely even or purely odd and is one-dimensional. Therefore, the notions of even roots and odd roots are well defined. The sets of all even roots, even positive roots, odd roots and odd isotropic roots are denoted by \( \Delta_{\overline{0}} \), \( \Delta_{\overline{0}}^+ \)), \( \Delta_{\overline{1}} \) and \( \Delta_{\otimes} \), respectively.

\end{definition}

\begin{definition}[\cite{musson2012lie,cheng2012dualities}]
We fix a Borel subalgebra \( \mathfrak{b}_{\overline{0}} \) of \( \mathfrak{g}_{\overline{0}} \). The set of all Borel subalgebras \( \mathfrak{b} \) of \( \mathfrak{g} \) that contain \( \mathfrak{b}_{\overline{0}} \) is denoted by \( \mathfrak{B(g)} \).

 The sets of positive roots, odd positive roots, and odd isotropic positive roots corresponding to \( \mathfrak{b} \) are denoted by \( \Delta^{\mathfrak{b}+} \), \( \Delta_{\overline{1}}^{\mathfrak{b}+} \), and \( \Delta_{\otimes}^{\mathfrak{b}+} \), respectively. The set of simple roots (basis) corresponding to \( \Delta^{\mathfrak{b}+} \) is denoted by \( \Pi^{\mathfrak{b}} \). We define \( \Pi_{\otimes}^{\mathfrak{b}} :=  \Pi^{\mathfrak{b}} \cap \Delta_{\otimes} \).
We define
\[
\Delta^{\operatorname{pure+}} := \bigcap_{\mathfrak{b} \in \mathfrak{B(g)}} \Delta^{\mathfrak{b}+},
\]
\[
\Delta_{\otimes}^{\operatorname{pure+}} := \bigcap_{\mathfrak{b} \in \mathfrak{B(g)}} \Delta_{\otimes}^{\mathfrak{b}+} = \Delta^{\operatorname{pure+}} \cap \Delta_{\otimes}.
\]
\end{definition}

\begin{theorem}[Odd reflection \cite{musson2012lie} 3.5]\label{2.3.oddref} 
    For \( \alpha \in \Pi_{\otimes}^{\mathfrak{b}} \), define \( r^{\mathfrak{b}}_{\alpha} \in \operatorname{Map}(\Pi^{\mathfrak{b}}, \Delta )\) by
    \[
        r^{\mathfrak{b}}_{\alpha}(\beta) = 
        \begin{cases}
            -\alpha & (\beta = \alpha), \\
            \alpha + \beta & (\alpha + \beta \in \Delta), \\
            \beta & (\text{otherwise}).
        \end{cases}
    \]
    for \( \beta \in \Pi^{\mathfrak{b}} \).
(When there is no risk of confusion, \( r_\alpha^\mathfrak{b} \) is abbreviated as \( r_\alpha \).)
   A Borel subalgebra \( r_{\alpha} \mathfrak{b} \in \mathfrak{B(g)} \) exists, with the corresponding basis given by 
\[
\Pi^{r_{\alpha} \mathfrak{b}} := \{ r^{\mathfrak{b}}_{\alpha}(\beta) \}_{\beta \in \Pi^{\mathfrak{b}}}.
\]

\end{theorem}
The linear transformation of \( \mathfrak{h}^* \) induced by an odd reflection does not necessarily map a Borel subalgebra to another Borel subalgebra.

The following is well-known:
\begin{proposition}[\cite{musson2012lie,cheng2012dualities}]\label{2.3.odd_kyoyaku}
Each pair of elements \( \mathfrak{b},\mathfrak{b'} \in \mathfrak{B(g)} \) due to transferred to each other by a finite number of odd reflections.
\end{proposition}

\begin{definition}\label{RBgdef}
    The edge-colored graph \( OR(\mathfrak{g}) \) is defined as follows:
    \begin{itemize}
        \item \textbf{Vertex set:} \( \mathfrak{B(g)} \).
        \item \textbf{Color set:} For a fixed \( \mathfrak{b} \in \mathfrak{B(g)} \), the set \( \Delta^{\mathfrak{b}+} \setminus \Delta^{\text{pure+}} \).
        \item \textbf{Edges and colors:} An edge is drawn between two vertices if they are related by an odd reflection. The edge is assigned a color corresponding to the unique \( \alpha \in \Delta^{\mathfrak{b}+} \setminus \Delta^{\text{pure+}} \) such that \( \alpha \) belongs to the positive root system of one vertex but not the other.
    \end{itemize}
    Since the positive root systems associated with different Borel subalgebras are in one-to-one correspondence, the structure of the edge-colored graph does not depend on the choice of \( \mathfrak{b} \).
\end{definition}

\begin{definition}[{\cite{bonfert2024weyl,serganova2011kac}}] \label{orderdrootbasis}

Let \( \mathfrak{b} \in \mathfrak{B(g)} \), and consider a total ordering \( \leq \) on \( \Pi^{\mathfrak{b}} \). We call the pair \( (\mathfrak{b}, \leq) \) an \emph{ordered root basis}. This ordering is denoted by 
\[
\Pi^{(\mathfrak{b}, \leq)} = \{ \alpha_1^{(\mathfrak{b}, \leq)}, \dots, \alpha_{\theta}^{(\mathfrak{b}, \leq)} \}.
\]

For a composition of odd reflections \( r_{\beta_t} \dots r_{\beta_1} \), we define the ordered root basis 
\[
r_{\beta_t} \dots r_{\beta_1}((\mathfrak{b}, \leq))
\]
by 
\[
\alpha_j^{r_{\beta_t} \dots r_{\beta_1}(\mathfrak{b}, \leq)} := r_{\beta_t} \dots r_{\beta_1}(\alpha_j^{(\mathfrak{b}, \leq)}).
\]

In this way, the ordered root bases are mapped to each other under odd reflections.
    
\end{definition}

\begin{definition} Recall \cref{orderdrootbasis}.
Given a fixed ordered root basis \( (\overline{\mathfrak{b}}, \overline{\leq}) \), we define \( E(\mathfrak{g}) \) as an edge-colored graph with the following structure:
\begin{itemize}
    \item \textbf{Vertex set \( V \):} Each vertex \( (\mathfrak{b}, \leq) \) represents an ordered root basis obtained from \( (\overline{\mathfrak{b}}, \overline{\leq}) \) through a finite sequence of odd reflections.
    
    \item \textbf{Color set :} The total orbdered set \( I \) as \cref{orderdrootbasis}.
    
   \item \textbf{Edges:} Draw an edge of color \( i \) between vertices that are related by an odd reflection corresponding to the \( i \)-th simple root. Additionally, assign a loop of color \( i \) at a vertex if the \( i \)-th simple root is non-isotropic for that vertex.
\end{itemize}
\end{definition}

We rely on the following result (see \cite[Definition 2.10]{bonfert2024weyl}, \cite[Corollary 2.14]{heckenberger2020hopf}, or \cite{azam2015classification}).
\begin{theorem}\label{BNmain}\label{2.3.main}
    Under the above settings, for each \( (\mathfrak{b}, \leq) \in V \), there exists a unique family of generalized Cartan matrices \( \{A^{(\mathfrak{b}, \leq)}\} \), such that the vertex labeling by this family of matrices makes \( E(\mathfrak{g}) \) a finite connected Cartan graph, and for each \( (\mathfrak{b}, \leq) \in V \), there is an additive bijection
    \[
        R^{(\mathfrak{b}, \leq)+} \simeq \Delta^{\mathfrak{b}+} \setminus 2 \Delta^{\mathfrak{b}+}
    \]
    given by mapping \( \alpha_i^{(\mathfrak{b}, \leq)} \mapsto \alpha_i^{\mathfrak{b}} \).

    We denote the Cartan graph constructed above by \( G(\mathfrak{g}) \).
\end{theorem}

\begin{corollary}[\cite{bonfert2024weyl} Remark 2.18]\label{futogo_icchi}
    If \( (\mathfrak{b}, \leq), (\mathfrak{b}, \leq') \in V \), then \( \leq = \leq' \).
    In particular, \( V \) can be identified with \(\mathfrak{B}(\mathfrak{g})\).
\end{corollary}

\begin{proof}
    This directly follows from (CG4) and \cref{BNmain}.
\end{proof}

\begin{theorem}\label{4.2.main}
    \( G(\mathfrak{g}) \) is path simply connected. Furthermore, \( \Delta \) in the sense of \cref{3.1delta} can be identified with the root system \( \Delta \) in the sense of \cref{A2delta}.
    
    As edge-colored graphs, \( OR(\mathfrak{g}) \) in the sense of \cref{RBgdef} is isomorphic to \( RB(G(\mathfrak{g})) \) in the sense of \cref{RBG}.

    In particular, \( OR(\mathfrak{g}) \) is a connected rainbow boomerang graph.
\end{theorem}

\begin{proof}
    This directly follows from \cref{BNmain} and \cref{A2main}.
\end{proof}

\begin{remark}
Here are a few remarks about the above facts:
    \begin{enumerate}
        \item By this construction, \( E(\mathfrak{g}) \) is indeed the exchange graph of \( G(\mathfrak{g}) \).
        
        \item The set \( R^{(\mathfrak{b}, \leq)} \) is a subset of \( (\mathbb{Z}^I)^{(\mathfrak{b}, \leq)} \), and \( \Delta \) is a subset of \( \mathfrak{h}^* \). We strictly distinguish between these two.

        \item The map \( s_i^{(\mathfrak{b}, \leq)} \) is a linear transformation from \( (\mathbb{Z}^I)^{(\mathfrak{b}, \leq)} \) to \( (\mathbb{Z}^I)^{r_i {(\mathfrak{b}, \leq)}} \), while the odd reflection \( r_i^{\mathfrak{b}} \) is a map from \( \Pi^{\mathfrak{b}} \) to \( \Delta \).

        \item By the above, \( G(\mathfrak{g}) \) does not depend on the choice of \( (\mathfrak{b}, \leq) \) and is uniquely determined by \( \mathfrak{g} \).
        
        \item For a vertex \( x \) in \( G(\mathfrak{g}) \), the automorphism group \( \operatorname{Aut}(x) \) can be identified with the Weyl group \( W \) (\cite[Proposition 2.15]{bonfert2024weyl}).
    \end{enumerate}
\end{remark}

\begin{example}
The general linear Lie superalgebra \( \mathfrak{gl}(m|n) \) is defined as the Lie superalgebra spanned by all \( E_{ij} \) with \( 1 \leq i, j \leq m+n \), under the supercommutator:
\[
[E_{ij}, E_{kl}] = \delta_{jk} E_{il} - (-1)^{|E_{ij}||E_{kl}|} \delta_{il} E_{kj},
\]
where \( |E_{ij}| = \overline{0} \) if \( E_{ij} \) acts within \( V_{\overline{0}} \) or \( V_{\overline{1}} \) (even), and \( |E_{ij}| = \overline{1} \) if it maps between \( V_{\overline{0}} \) and \( V_{\overline{1}} \) (odd).

The Cartan subalgebra \( \mathfrak{h} \) is given by \( \mathfrak{h} = \bigoplus k E_{ii} \).

Let \( E_{ii} \) be associated with dual basis elements \( \varepsilon_i \) for \( 1 \leq i \leq m+n \). Then we have \( \mathfrak{g}_{\varepsilon_i - \varepsilon_j} = k E_{ij} \).

Define \( \delta_i = \varepsilon_{m+i} \) for \( 1 \leq i \leq n \). The sets of roots are as follows:
\[
\Delta_{\overline{0}} = \{ \varepsilon_i - \varepsilon_j, \delta_i - \delta_j \mid i \neq j \},
\]
\[
\Delta_{\overline{1}} = \{ \varepsilon_i - \delta_j \mid 1 \leq i \leq m, \, 1 \leq j \leq n \}.
\]

For the even part \( \mathfrak{g}_{\overline{0}} = \mathfrak{gl}(m) \oplus \mathfrak{gl}(n) \), we fix the standard Borel subalgebra \( \mathfrak{b}_{\overline{0}} \) as:
\[
\mathfrak{b}_{\overline{0}} = \bigoplus_{1 \leq i \leq j \leq m} k E_{ij} \oplus \bigoplus_{m+1 \leq i \leq j \leq n} k E_{ij}.
\]

We assume that the Borel subalgebras we consider all contain \( \mathfrak{b}_{\overline{0}} \).  
Such Borel subalgebras are in bijection with Young diagrams fitting inside an \( m \times n \) rectangle, and the associated odd reflection graph is isomorphic to the finite Young lattice.  
We denote a Young diagram by expressions such as \( (4\,2^2\,1) \), and we write the empty diagram as \( \emptyset \), which corresponds precisely to the standard Borel subalgebra.  
For further details, see \cite{Hirota2025}.

According to \cite{bonfert2024weyl}, fixing the total order determined by
 identifying \( \varepsilon_i - \varepsilon_{i+1} \) with \( \operatorname{orb}(\alpha_i^{\emptyset}) \).

\( E(\mathfrak{gl}(m|n)) \) (excluding loops) is defined as an edge-colored graph with the following structure \cite{bonfert2024weyl}:
\begin{itemize}
    \item \textbf{Vertex set :} \( V = \mathfrak{B}(\mathfrak{g}) = P_{m \times n} \) (Young diagrams fitting in a m×n rectangle.)
    
    \item \textbf{Color set :} \( I = \{ 1, 2, \dots, m+n-1 \} \);
    
    \item \textbf{Edges:} There is an edge of color \( i \) between vertices \( \mathfrak{b}_1 \) and \( \mathfrak{b}_2 \) if and only if \( \mathfrak{b}_1 \) and \( \mathfrak{b}_2 \) are related by adding or subtracting a box at coordinates \( (x, y) \) in French notation, with \( x - y + m = i \).
\end{itemize}

Furthermore, the graph \( G(\mathfrak{gl}(m|n)) \) is the labeled graph obtained by labeling each vertex \( \mathfrak{b} \) of \( E(\mathfrak{gl}(m|n)) \) with \( A^{\mathfrak{b}} = A_{m+n-1} \).
\end{example}
\begin{example}
In \( \mathfrak{gl}(2|1) \), we have the following identifications:
\[
\varepsilon_1 - \varepsilon_2 \leftrightarrow \operatorname{orb}(\alpha^{\emptyset}_1) = \bigl\{ \alpha^{\emptyset}_1, \alpha^{(1)}_1 + \alpha^{(1)}_2, \alpha^{(1^2)}_2 \bigr\}.
\]
\[
\varepsilon_1 - \delta_1 \leftrightarrow \operatorname{orb}(\alpha^{\emptyset}_1 + \alpha^{\emptyset}_2) = \bigl\{ \alpha^{\emptyset}_1 + \alpha^{\emptyset}_2, \alpha^{(1)}_1, -\alpha^{(1^2)}_1 \bigr\}.
\]
\[
\varepsilon_2 - \delta_1 \leftrightarrow \operatorname{orb}(\alpha^{\emptyset}_2) = \bigl\{ \alpha^{\emptyset}_2, -\alpha^{(1)}_2, -\alpha^{(1^2)}_1 - \alpha^{(1^2)}_2 \bigr\}.
\]
\end{example}

\begin{example}

The exchange graph \( E(\mathfrak{gl}(3|2)) \) (excluding loops) is as follows. 

\hspace{-10mm}
\begin{tikzpicture}[scale=1.5]
     \node (A) at (0, 0) {\(\emptyset\)};
    \node (B) at (1, 0) {\(\begin{ytableau} ~ \end{ytableau}\)};

    \node (C) at (2.5, 0.75) {\(\begin{ytableau} ~ & ~ \\ \end{ytableau}\)};
    \node (D) at (4, 0.75) {\(\begin{ytableau} ~ \\ ~ & ~ \end{ytableau}\)};
    \node (E) at (5.5, 0.75) {\(\begin{ytableau} ~ & ~ \\ ~ & ~ \end{ytableau}\)};
    \node (F) at (7, 0) {\(\begin{ytableau} ~ \\ ~ & ~ \\ ~ & ~  \end{ytableau}\)};
        \node (G) at (9, 0) {\(\begin{ytableau} ~ & ~ \\  ~ & ~  \\ ~ & ~ \end{ytableau}\)};
    \node (H) at (2.5, -0.75) {\(\begin{ytableau} ~ \\ ~ \end{ytableau}\)};
    \node (I) at (4, -0.75) {\(\begin{ytableau} ~ \\ ~ \\ ~ \end{ytableau}\)};
    \node (J) at (5.5, -0.75) {\(\begin{ytableau} ~ \\ ~ \\ ~ & ~ \end{ytableau}\)};
    
\draw (A) -- node[above] {3} (B);
\draw (B) -- node[above] {4} (C);
\draw (C) -- node[above] {2} (D);
\draw (D) -- node[above] {3} (E);
\draw (E) -- node[above] {1} (F);
\draw (F) -- node[above] {2} (G);
\draw (B) -- node[above] {2} (H);
\draw (D) -- node[above] {4} (H);
\draw (D) -- node[above] {1} (J);
\draw (F) -- node[above] {3} (J);
\draw (H) -- node[above] {1} (I);
\draw (I) -- node[above] {4} (J);
\end{tikzpicture}.

The odd reflection graph \( OR(\mathfrak{gl}(3|2)) \) is as follows.

\begin{center}

\begin{tikzpicture}[scale=1.5]
     \node (A) at (0, 0) {\(\emptyset\)};
    \node (B) at (1, 0) {\(\begin{ytableau} ~ \end{ytableau}\)};

    \node (C) at (2.5, 0.75) {\(\begin{ytableau} ~ & ~ \\ \end{ytableau}\)};
    \node (D) at (4, 0.75) {\(\begin{ytableau} ~ \\ ~ & ~ \end{ytableau}\)};
    \node (E) at (5.5, 0.75) {\(\begin{ytableau} ~ & ~ \\ ~ & ~ \end{ytableau}\)};
    \node (F) at (7, 0) {\(\begin{ytableau} ~ \\ ~ & ~ \\ ~ & ~  \end{ytableau}\)};
        \node (G) at (9, 0) {\(\begin{ytableau} ~ & ~ \\  ~ & ~  \\ ~ & ~ \end{ytableau}\)};
    \node (H) at (2.5, -0.75) {\(\begin{ytableau} ~ \\ ~ \end{ytableau}\)};
    \node (I) at (4, -0.75) {\(\begin{ytableau} ~ \\ ~ \\ ~ \end{ytableau}\)};
    \node (J) at (5.5, -0.75) {\(\begin{ytableau} ~ \\ ~ \\ ~ & ~ \end{ytableau}\)};
    
\draw (A) -- node[above] {(1,1)} (B);
\draw (B) -- node[above] {(2,1)} (C);
\draw (C) -- node[above] {(1,2)} (D);
\draw (D) -- node[above] {(2,2)} (E);
\draw (E) -- node[above] {(1,3)} (F);
\draw (F) -- node[above] {(2,3)} (G);
\draw (B) -- node[above] {(1,2)} (H);
\draw (D) -- node[above] {(2,1)} (H);
\draw (D) -- node[above] {(1,3)} (J);
\draw (F) -- node[above] {(2,2)} (J);
\draw (H) -- node[above] {(1,3)} (I);
\draw (I) -- node[above] {(2,1)} (J);
\end{tikzpicture}.
\end{center}

\end{example}

 \begin{example}\label{2.3.d21}

Let \( \mathfrak{g} = D(2, 1; \alpha) \). See \cite{cheng2019character} for more information on this type of Lie superalgebra.

The vector space \( \mathfrak{h}^* \) has an basis \( \{ \delta, \varepsilon_1, \varepsilon_2 \} \).

The sets of roots are as follows:
\[
\Delta_{\overline{0}} = \{ \pm 2 \delta, \pm 2 \varepsilon_1, \pm 2 \varepsilon_2 \}
\]
\[
\Delta_{\overline{1}} = \Delta_{\otimes} = \{ \pm (\delta - \varepsilon_1 - \varepsilon_2), \pm (\delta + \varepsilon_1 - \varepsilon_2), \pm (\delta - \varepsilon_1 + \varepsilon_2), \pm (\delta + \varepsilon_1 + \varepsilon_2) \}
\]

The exchange graph \( E(D(2,1;\alpha)) \) (excluding loops) is described as follows.
\begin{center}

\begin{tikzpicture}
    \node (b1) at (0,2) {\(\mathfrak{b}_1\)};
    \node (b2) at (-2,0) {\(\mathfrak{b}_2\)};
    \node (b3) at (0,0) {\(\mathfrak{b}_3\)};
    \node (b4) at (2,0) {\(\mathfrak{b}_4\)};
    
    \draw (b1) -- node[left] {2} (b3);
    \draw (b2) -- node[above] {1} (b3);
    \draw (b3) -- node[above] {3} (b4);
\end{tikzpicture}
\end{center}
The Cartan graph \( G(D(2,1;\alpha)) \) is defined as follows.

\[
\begin{array}{cc}
A^{\mathfrak{b}_1} = 
\begin{pmatrix} 
2 & -1 & 0 \\ 
-1 & 2 & -1 \\ 
0 & -1 & 2 
\end{pmatrix},\quad
&
A^{\mathfrak{b}_2} = 
\begin{pmatrix} 
2 & -1 & -1 \\ 
-1 & 2 & 0 \\ 
-1 & 0 & 2 
\end{pmatrix},
\\[10pt]
A^{\mathfrak{b}_3} = 
\begin{pmatrix} 
2 & -1 & -1 \\ 
-1 & 2 & -1 \\ 
-1 & -1 & 2 
\end{pmatrix},\quad
&
A^{\mathfrak{b}_4} = 
\begin{pmatrix} 
2 & 0 & -1 \\ 
0 & 2 & -1 \\ 
-1 & -1 & 2 
\end{pmatrix}.
\end{array}
\]

The corresponding positive root systems for each vertex are:

\[
\begin{aligned}
R^{\mathfrak{b}_1 +} &= \{ \alpha_1^{\mathfrak{b}_1}, \alpha_1^{\mathfrak{b}_1} + \alpha_2^{\mathfrak{b}_1}, \alpha_1^{\mathfrak{b}_1} + \alpha_2^{\mathfrak{b}_1} + \alpha_3^{\mathfrak{b}_1}, \alpha_1^{\mathfrak{b}_1} + 2\alpha_2^{\mathfrak{b}_1} + \alpha_3^{\mathfrak{b}_1}, \alpha_2^{\mathfrak{b}_1}, \alpha_2^{\mathfrak{b}_1} + \alpha_3^{\mathfrak{b}_1}, \alpha_3^{\mathfrak{b}_1} \} \\[15pt]
R^{\mathfrak{b}_2 +} &= \{ \alpha_2^{\mathfrak{b}_2}, \alpha_2^{\mathfrak{b}_2} + \alpha_1^{\mathfrak{b}_2}, \alpha_2^{\mathfrak{b}_2} + \alpha_1^{\mathfrak{b}_2} + \alpha_3^{\mathfrak{b}_2}, \alpha_2^{\mathfrak{b}_2} + 2\alpha_1^{\mathfrak{b}_2} + \alpha_3^{\mathfrak{b}_2}, \alpha_1^{\mathfrak{b}_2}, \alpha_1^{\mathfrak{b}_2} + \alpha_3^{\mathfrak{b}_2}, \alpha_3^{\mathfrak{b}_2} \} \\[15pt]
R^{\mathfrak{b}_3 +} &= \{ \alpha_1^{\mathfrak{b}_3}, \alpha_1^{\mathfrak{b}_3} + \alpha_2^{\mathfrak{b}_3}, \alpha_1^{\mathfrak{b}_3} + \alpha_3^{\mathfrak{b}_3}, \alpha_1^{\mathfrak{b}_3} + \alpha_2^{\mathfrak{b}_3} + \alpha_3^{\mathfrak{b}_3}, \alpha_2^{\mathfrak{b}_3}, \alpha_2^{\mathfrak{b}_3} + \alpha_3^{\mathfrak{b}_3}, \alpha_3^{\mathfrak{b}_3} \} \\[15pt]
R^{\mathfrak{b}_4 +} &= \{ \alpha_1^{\mathfrak{b}_1}, \alpha_1^{\mathfrak{b}_1} + \alpha_3^{\mathfrak{b}_1}, \alpha_1^{\mathfrak{b}_1} + \alpha_3^{\mathfrak{b}_1} + \alpha_2^{\mathfrak{b}_1}, \alpha_1^{\mathfrak{b}_1} + 2\alpha_3^{\mathfrak{b}_1} + \alpha_2^{\mathfrak{b}_1}, \alpha_3^{\mathfrak{b}_1}, \alpha_3^{\mathfrak{b}_1} + \alpha_2^{\mathfrak{b}_1}, \alpha_2^{\mathfrak{b}_1} \}
\end{aligned}
\]

Fixing a suitable total order,  for example, the following correspondences hold:
\[
2 \varepsilon_1 \leftrightarrow \operatorname{orb}(\alpha_1^{\mathfrak{b}_1}), \quad 
\delta - \varepsilon_1 - \varepsilon_2 \leftrightarrow \operatorname{orb}(\alpha_2^{\mathfrak{b}_1}), \quad 
2 \varepsilon_2 \leftrightarrow \operatorname{orb}(\alpha_3^{\mathfrak{b}_1}),
\]
\[
\delta + \varepsilon_1 + \varepsilon_2 \leftrightarrow \operatorname{orb}(\alpha_1^{\mathfrak{b}_1} + \alpha_2^{\mathfrak{b}_1} + \alpha_3^{\mathfrak{b}_1}),
\quad 
2 \delta \leftrightarrow \operatorname{orb}(\alpha_1^{\mathfrak{b}_1} + 2 \alpha_2^{\mathfrak{b}_1} + \alpha_3^{\mathfrak{b}_1}).
\]

we  also note that
\[
\Delta^{\text{pure}+} = \{ 2\delta, 2\varepsilon_1, 2\varepsilon_2, \delta + \varepsilon_1 + \varepsilon_2 \} 
, \quad
\Delta_\otimes^{\text{pure}+} = \{ \delta + \varepsilon_1 + \varepsilon_2 \}.
\]

\end{example}

  For other types of \( G(\mathfrak{g}) \), see \cite{musson2012lie,bonfert2024weyl,andruskiewitsch2017finite}.

%% file: A.4.tex
\subsection{Nichols algebras of diagonal type}
The braided monoidal category \( {}^G_G\mathcal{YD} \) (resp. \( {}^G_G\mathcal{YD}^\mathrm{fd} \)) of Yetter-Drinfeld modules (resp. finite-dimensional Yetter-Drinfeld modules) over a group \( G \), as well as the Nichols algebra over them, are discussed in detail in \cite{etingof2015tensor,heckenberger2020hopf}.

\begin{definition}\label{def_dim}
\(G \text{-grVec} ^\mathrm{fd}\) denotes the tensor category of finite-dimensional \( G \)-graded vector spaces.

Consider the forgetful functor as underlying tensor categories (but not as braided tensor categories!):
\[
F: {}_G^G\mathcal{YD} ^\mathrm{fd}\cong G\text{-grVec}^\mathrm{fd} \to \text{Vec}^\mathrm{fd}.
\]

Define
\[
\dim V := \dim_k F(V).
\]

Let \( \theta \in \mathbb{N} \), and set \( I = \{1, 2, \ldots, \theta\} \). We denote \( \{\alpha_1, \ldots, \alpha_\theta\} \) as the canonical \( \mathbb{Z} \)-basis of \( \mathbb{Z}^\theta \).

For a bicharacter \( \mathfrak{q}(- , -) : \mathbb{Z}^\theta \times \mathbb{Z}^\theta \to k^\times \), there exists a direct sum 
\[
V = kx_1 \oplus \dots \oplus kx_\theta \in {^{\mathbb{Z}^\theta}_{\mathbb{Z}^\theta} \mathcal{YD}}
\]
of \( \theta \) one-dimensional Yetter-Drinfeld modules such that the following holds:
\[
c_{V,V}(x_i \otimes x_j) = \mathfrak{q}(\alpha_i, \alpha_j) x_j \otimes x_i \quad i,j \in I
\]
We write \( q_{ij} = \mathfrak{q}(\alpha_i, \alpha_j) \).

This \( V \)'s Nichols algebra is denoted by \( B_\mathfrak{q} \). Such Nichols algebras are called of diagonal type. The algebra and coalgebra structures of \( B_\mathfrak{q} \) are fully determined by the braiding matrix \( (q_{ij}) \). The number \( \theta \) is called the rank of \( B_\mathfrak{q} \).

\( B_\mathfrak{q} \) inherits a natural \( \mathbb{Z}_{\geq 0}^\theta \)-grading \( \deg x_i := \alpha_i \). This grading is compatible with both the algebra and coalgebra structures of \( B_\mathfrak{q} \), by the construction of \( B_\mathfrak{q} \) as a Nichols algebra.
By these definitions, \( B_\mathfrak{q} \) is a bimonoid object in \( \mathbb{Z}^\theta_{\geq 0}\text{-gr} {}_{\mathbb{Z}^\theta}^{\mathbb{Z}^\theta}\mathcal{YD} \).

\end{definition}

\begin{example}[Classification of rank 1 Nichols algebras \cite{heckenberger2020hopf} Example 1.10.1]\label{Ni_rank1} When \( \theta = 1 \), a bicharacter \(\mathfrak{q}\) can be identified with an element \( q \in k^\times \).
The graded algebra \( B_\mathfrak{q} \) is classified as follows:
\[
B_\mathfrak{q} \simeq
\begin{cases}
k[x]/(x^{\operatorname{ord}q}) &  1 < \operatorname{ord}q < \infty, \\
k[x] & q = 1 \text{ or } \operatorname{ord}q = \infty.
\end{cases}
\]
\end{example}

\begin{example}[\cite{heckenberger2020hopf} Theorem 16.2.5]
Let \( A = (d_i a_{ij}) \) be a symmetrized generalized Cartan matrix. If \( q_{ij} = q^{d_i a_{ij}} \) and \( q \) is not a root of unity, then we have \( B_\mathfrak{q} \simeq U_{q}^+(\mathfrak{g}_A) \).

Here, \( U_{q}^+(\mathfrak{g}_A) \) represents the positive part of the quantum group associated with Kac-Moody Lie algebra \( \mathfrak{g}_A \). 
This result was first proven by Lusztig in \cite{lusztig2010introduction} in the case of finite type.

\end{example}

Using the theory of Lyndon words, a PBW-type basis for a Nichols algebra can be constructed.

\begin{theorem}[\cite{andruskiewitsch2017finite} 2.6, \cite{kharchenko2015quantum}] \label{Ni_PBW}
For a bicharacter \( \mathfrak{q} \), there exists a totally ordered set \( (S, \leq) \) such that for each \( s \in S \), there exists a homogeneous element \( X_s \in B_\mathfrak{q} \) satisfying:
\[
\left\{ X_{l_1}^{m_1} \cdots X_{l_k}^{m_k} \mid k \geq 0, \, l_1 < \cdots < l_k \in S, \, 0 \leq m_i < \operatorname{ord}\mathfrak{q}(\deg X_{l_i}, \deg X_{l_i}) \right\}
\]
is a basis for \( B_\mathfrak{q} \).
\end{theorem}

\begin{proposition}[\cite{andruskiewitsch2008nichols} Lemma 2.18]\label{R_q}
If \( \#R_\mathfrak{q}^+ < \infty \), we define \( R_\mathfrak{q}^+ = \{ \deg(X_s) \mid s \in S \} \subset \mathbb{Z}_{\geq 0}^\theta \). Then \( R_\mathfrak{q}^+ \) does not depend on the choice of the ordered set \( S \).
\end{proposition}

\begin{corollary}
If \( \#R_\mathfrak{q}^+ < \infty \), then there is a \( \mathbb{Z}_{\geq 0}^\theta \)-graded Yetter-Drinfeld module isomorphism
\[
B_\mathfrak{q} \simeq \bigotimes_{\alpha \in R_\mathfrak{q}^+} B(k X_\alpha).
\]
\end{corollary}

\begin{proof}
    This follows from \cref{Ni_rank1} and \cref{Ni_PBW}.
\end{proof}

\begin{corollary}
\( \dim B_\mathfrak{q} < \infty \) if and only if \( \#R_\mathfrak{q}^+ < \infty \) and \(1<\operatorname{ord}\mathfrak{q}(\alpha, \alpha)  < \infty \) for all \( \alpha \in R_\mathfrak{q}^+ \).
\end{corollary}

%% file: A.5.tex
\subsection{Lusztig autmorphisms of small quantum groups}

In this subsection, we introduce the algebras in which we are interested in. We will follow \cite{vay2023linkage}. 

\begin{definition}
We denote \(\tilde{U}_\mathfrak{q}\) as the Hopf algebra generated by the symbols \( K_i, K_i^{-1}, L_i, L_i^{-1}, E_i \), and \( F_i \), with \( i \in I \), subject to the relations:
    \[
    K_iE_j = q_{ij}E_jK_i, \quad L_iE_j = q_{ji}^{-1}E_jL_i,
    \]
    \[
    K_iF_j = q_{ij}^{-1}F_jK_i, \quad L_iF_j = q_{ji}F_jL_i,
    \]
    \[
    E_iF_j - F_jE_i = \delta_{i,j}(K_i - L_i),
    \]
    \[
    XY = YX, \quad K_iK_i^{-1} = L_iL_i^{-1} = 1,
    \]
    for all \( i, j \in I \) and \( X, Y \in \{K_i^{\pm 1}, L_i^{\pm 1} \mid i \in I\} \).

The counit \( \varepsilon: U_{\mathfrak{q}} \to k \) is defined as
\[
\varepsilon(K_i^{\pm 1}) = \varepsilon(L_i^{\pm 1}) = \varepsilon(E_i) = \varepsilon(F_i) = 0 \quad \text{for all } i \in I.
\]

Let \( \tau \) be the algebra antiautomorphism of \( \tilde{U}_q \) defined by
\[
\tau(K_i) = K_i, \quad \tau(L_i) = L_i, \quad \tau(E_i) = F_i, \quad \text{and} \quad \tau(F_i) = E_i
\]
for all \( i \in \mathbb{I} \).

Let \( J_\mathfrak{q} \) be the defining relation of the Nichols algebra of diagonal type, generated by \( E_i \), which is determined by the braiding matrix \( (q_{ij}) \).

Let \( {U}_\mathfrak{q} \) be the Hopf algebra obtained by quotienting \( \tilde{U}_q \) by \( J_\mathfrak{q} \) and \( \tau(J_\mathfrak{q}) \).

We have that $U_\mathfrak{q} = \bigoplus_{\mu \in \mathbb{Z}^\theta} (U_\mathfrak{q})_\mu$ is a $\mathbb{Z}^\theta$-graded Hopf algebra with
\[
\deg E_i = -\deg F_i = \alpha_i \quad \text{and} \quad \deg K_i^{\pm 1} = \deg L_i^{\pm 1} = 0 \quad \forall i \in I.
\]

The multiplication of $U_\mathfrak{q}$ induces a linear isomorphism
\[
U_\mathfrak{q}^- \otimes U_\mathfrak{q}^0 \otimes U_\mathfrak{q}^+ \cong U_\mathfrak{q},
\]
where
\[
U_\mathfrak{q}^+ = k\langle E_i \mid i \in I \rangle \cong \mathfrak{B}_\mathfrak{q}, \quad 
U_\mathfrak{q}^0 = k\langle K_i^{\pm 1}, L_i^{\pm 1} \mid i \in I \rangle, \quad 
U_\mathfrak{q}^- = k\langle F_i \mid i \in I \rangle.
\]
are \( \mathbb{Z}^\theta \)-graded subalgebras of $U_\mathfrak{q}$. We remark that \( U_\mathfrak{q}^0 \cong k(\mathbb{Z}^\theta \times \mathbb{Z}^\theta) \).

\end{definition}

\begin{remark}
We do not require an explicit presentation of the defining relations of \( \mathcal{B}_\mathfrak{q} \)  or the coproduct and antipode structures. All we need is the following remarkable Lusztig automorphism, which creates distinctions from the highest weight theory of more general Hopf algebras with trianglar decomposition \cite{vay2018hopf}.

There are variations of what is called "small quantum groups". However, as shown in \cite[Corollary 8.17]{vay2023linkage} , results on our algebra \( U_\mathfrak{q} \) can be applied to a broad class of small quantum groups, such as those discussed in \cite{lusztig1990quantum} or \cite{laugwitz2023finite}.

\end{remark}

\begin{theorem}[\cite{heckenberger2006weyl,heckenberger2010lusztig,azam2015classification}]
\label{Ni-Lus}
Let \( \bar{\mathfrak{q}} \) be a bicharacter such that \( B_{\bar{\mathfrak{q}}} \) is finite-dimensional. Then, there exists a simply connected finite Cartan graph \( G[\bar{\mathfrak{q}}] \) with a vertex set \( V(G[\bar{\mathfrak{q}}]) \) consisting of bicharacters with finite-dimensional Nichols algebras. For each vertex \( \mathfrak{q} \), there is an additive bijection between \( R_\mathfrak{q}^+ \) (in the sense of \cref{R_q}) and \( R^{\mathfrak{q}+} \) (in the sense of \cref{R^q}).

Moreover, for \( w \in \operatorname{Hom}_{W(G(\bar{\mathfrak{q}}))}(\mathfrak{q}_1, \mathfrak{q}_2) \), there exists an algebra isomorphism 
\[
T_w : U_{\mathfrak{q}_1} \to U_{\mathfrak{q}_2}
\]
satisfying
\[
T_w((U_{\mathfrak{q}_1})_\alpha) = (U_{\mathfrak{q}_2})_{w\alpha}, \quad \text{for any } \alpha \in \mathbb{Z}^\theta.
\]
\end{theorem}

\begin{example}
The representation theory of a small quantum group corresponding to the rank 2 Nichols algebra \( B_\mathfrak{q} \) of type \( \text{ufo}(7) \) is described in detail (\cite{andruskiewitsch2018simple}). For this \( \mathfrak{q} \), \( G[\mathfrak{q}]\) is the Cartan graph given in \cref{w.5}, and it is known that such objects do not arise from (modular) contragredient Lie (super) algebras. The \( \mathbb{Z}^2 \)-degree of the PBW basis of \( B_\mathfrak{q} \) can be easily read from the frieze pattern in \cref{w.5}.
\end{example}

\begin{remark}
 \( G[\mathfrak{q}]\) is the simply connected cover of the small Cartan graph of \( \mathfrak{q} \) in the sense of \cite{heckenberger2020hopf}. For the Nichols algebra \( \mathfrak{B}_{\mathfrak{q}} \) of super type with the same Weyl groupoid as the basic Lie superalgebra \( \mathfrak{g} \), we have \( G[\mathfrak{q}] = SC(G(\mathfrak{g})) \).
\end{remark}

\begin{definition}
For a bicharacter \( \mathfrak{q} \) with finite-dimensional Nichols algebra, we define the rainbow boomerang graph \( RB[\mathfrak{q}] := RB(G[\mathfrak{q}]) \) (\cref{RBG}).
Note that \( G[\mathfrak{q}]\) is simply connected, so it is trivially path simply connected.
\end{definition}

%% file: A.6.tex
\subsection{Odd Verma's theorem}

For \( \alpha = n_1\alpha_1 + \cdots + n_\theta \alpha_\theta \in \mathbb{Z}^\theta \), we set
\[
K_\alpha = K_1^{n_1} \cdots K_\theta^{n_\theta} \quad \text{and} \quad L_\alpha = L_1^{n_1} \cdots L_\theta^{n_\theta}.
\]
In particular, \( K_{\alpha_i} = K_i \) for \( i \in I \).

\begin{definition}

Fix a bicharacter \(
\bar{\mathfrak{q}} : \mathbb{Z}^\theta \times \mathbb{Z}^\theta \to k^\times.
\) If \( w \in \operatorname{Hom}_{W(G[\bar{\mathfrak{q}}])}(\mathfrak{q}, \bar{\mathfrak{q}}) \), then the triangular decomposition of \( U_\mathfrak{q} \) induces a new triangular decomposition on \( U_{\bar{\mathfrak{q}}} \). Explicitly,
\[
T_w(U_\mathfrak{q}^-) \otimes U_{\bar{\mathfrak{q}}}^0 \otimes T_w(U_\mathfrak{q}^+) \cong U_{\bar{\mathfrak{q}}},
\]
since \( T_w(U_\mathfrak{q}^0) = U_{\bar{\mathfrak{q}}}^0 \).

Given \( \lambda \in \mathbb{Z}^\theta \), we consider \( k_\lambda^\mathfrak{q} = k v_\lambda^\mathfrak{q} \) as a \( \mathbb{Z}^\theta \)-graded \( U_{\bar{\mathfrak{q}}}^0 T_w(U_\mathfrak{q}^+) \)-module concentrated in degree \( \lambda \) with the action
\[
K_\alpha L_\beta u v_\lambda^\mathfrak{q} = \varepsilon(u) \frac{\mathfrak{\bar{q}}(\alpha, \lambda)}{\mathfrak{\bar{q}}(\lambda, \beta)} v_\lambda^\mathfrak{q}, \quad \forall K_\alpha L_\beta \in U_{\bar{\mathfrak{q}}}^0, \, u \in T_w(U_\mathfrak{q}^+).
\]
Osing this, we introduce the \( \mathbb{Z}^\theta \)-graded \( U_{\bar{\mathfrak{q}}} \)-module
\[
M^\mathfrak{\bar{q}}(\lambda) = U_{\bar{\mathfrak{q}}} \otimes_{U_{\bar{\mathfrak{q}}}^0 T_w(U_\mathfrak{q}^+)} k_\lambda^\mathfrak{q}.
\]

Note that for all \( v \in (U_\mathfrak{\bar{q}})_\lambda \), we can compute:
\[
K_\alpha L_\beta v = \frac{\mathfrak{\bar{q}}(\alpha, \lambda)}{\mathfrak{\bar{q}}(\lambda, \beta)} v K_\alpha L_\beta,
\]
(see \cite{vay2023linkage}[((4.3)).

We also define an analog of the Weyl vector. (It differs from the one for \cite{vay2023linkage} by a factor of \(-1\).) Specifically, we define:
\[
\rho^\mathfrak{q} := -\frac{1}{2} \sum_{\beta \in R_\mathfrak{q}^+} (\operatorname{ord}\mathfrak{q}(\beta, \beta) - 1) \beta.
\]
\end{definition}

\begin{remark}
Instead of our special \( k_\lambda^\mathfrak{q} \), we could consider a more general situation. However, by \cite[Proposition 5.5]{vay2023linkage}, all blocks are equivalent to the block containing our Verma module (the principal block). Thus, for simplicity, we restrict our discussion to this case. 

Our Verma module corresponds to so called a type I representation when \( U_{\bar{\mathfrak{q}}} \) is of the classical Drinfeld-Jimbo type.
\end{remark}

We consider the category \( \mathbb{Z}^\theta\text{-gr}(U_{\bar{\mathfrak{q}}}\text{-Mod}) \), where morphisms respects this \( \mathbb{Z}^\theta\)-grading. (This is the module category of a monoid object in the category of \( \mathbb{Z}^\theta\text{-graded vector spaces} \) in the sense of \cite{etingof2015tensor}.)

Let \( M = \bigoplus_{\nu \in \mathbb{Z}^\theta} M_\nu \) be a \( \mathbb{Z}^\theta \)-graded vector space. The formal character of \( M \) is defined as:
\[
\mathrm{ch} \, M = \sum_{\mu \in \mathbb{Z}^\theta} \mathrm{dim}M_\mu e^\mu.
\]
The following can be shown in the same way as in \cite{Hirota2025}.
Note that
\[
\mathrm{ch} \, U_\mathfrak{q}^- = \prod_{\beta \in R_\mathfrak{q}^+} \frac{1 - e^{-\operatorname{ord}\mathfrak{\bar{q}}(\beta, \beta)\beta}}{1 - e^{-\beta}} = \prod_{\beta \in R_\mathfrak{q}^+} \left( 1 + e^{-\beta} + \cdots + e^{(1-\operatorname{ord}\mathfrak{\bar{q}}(\beta, \beta))\beta} \right).
\]

\begin{proposition} \cite[Lemma 6.1]{vay2023linkage}
Let \( w \in \operatorname{Hom}_{W(G[\bar{\mathfrak{q}}])}(\mathfrak{q}, \bar{\mathfrak{q}}) \) and \( w' \in \operatorname{Hom}_{W(G[\bar{\mathfrak{q}}])}(\mathfrak{q}', \bar{\mathfrak{q}}) \). For a pair of vertices \( \mathfrak{q}, \mathfrak{q}' \) in \( G[\bar{\mathfrak{q}}] \) and \( \lambda \in \mathfrak{h}^* \), the following statements hold:
\begin{enumerate}
    \item \( \operatorname{ch} M^{\mathfrak{q}}(\lambda - w\rho^{\mathfrak{q}}) = \operatorname{ch} M^{\mathfrak{q}'}(\lambda - w'\rho^{\mathfrak{q}'}) \).
    \item \( \dim M^{\mathfrak{q}'}(\lambda - w'\rho^{\mathfrak{q}'})_{\lambda - w\rho^{\mathfrak{q}}} = 1 \).
    \item \label{(5.2.dimhom} \( \dim \operatorname{Hom}(M^{\mathfrak{q}}(\lambda - w\rho^{\mathfrak{q}}), M^{\mathfrak{q}'}(\lambda - w'\rho^{\mathfrak{q}'})) = 1 \).
\end{enumerate}
\end{proposition}

\begin{definition}

For \( \lambda \in \mathbb{Z}^\theta \), we denote a nonzero homomorphism from \( M^{\mathfrak{q}}(\lambda - w\rho^{\mathfrak{q}}) \) to \( M^{\mathfrak{q}'}(\lambda - w'\rho^{\mathfrak{q}'}) \) by \( \psi_{\lambda}^{\mathfrak{q} \mathfrak{q}'} \). Let the highest weight vector of \( M^{\mathfrak{q}}(\lambda) \) be \( v_{\lambda}^{\mathfrak{q}} \).
\end{definition}

\begin{definition}
    For $q \in k^\times$ and $n \in \mathbb{N}$, we recall the quantum numbers
\[
(n)_q = \sum_{j=0}^{n-1} q^j 
\].
\end{definition}

\begin{proposition}\cite[Section 7.1]{vay2023linkage}
  For \( \lambda \in \mathbb{Z}^\theta \), exactly one of the following holds:
    \begin{enumerate}
        \item \( \psi_{\lambda}^{r_i \mathfrak{q}, \mathfrak{q}} \) and \( \psi_{\lambda}^{\mathfrak{q}, r_i \mathfrak{q}} \) are isomorphisms.  
        
        \item 
        
\[
\psi_{\lambda}^{r_i \mathfrak{q}, \mathfrak{q}} \circ \psi_{\lambda}^{\mathfrak{q}, r_i \mathfrak{q}} = 
\psi_{\lambda}^{\mathfrak{q}, r_i \mathfrak{q}} \circ \psi_{\lambda}^{r_i \mathfrak{q}, \mathfrak{q}} = 0.
\]
    \end{enumerate}
\end{proposition}

\begin{proof} By the Lusztig automorphism, the case of positive roots can be reduced to that of simple roots, so we may assume \( \mathfrak{q} = \overline{\mathfrak{q}} \). (This is formalized in \cite[Section 7.3]{vay2023linkage}[Section 7.3] by constructing a suitable category equivalence.)

From the defining relations, we can compute:
\begin{equation}
\begin{split}
E_i F_i^n &= F_i^n E_i + F_i^{n-1} \left( (n)_{q_{ii}^{-1}} K_{i} - (n)_{q_{ii}} L_{i} \right), \\
F_i E_i^n &= E_i^n F_i + E_i^{n-1} \left( (n)_{q_{ii}^{-1}} L_{i} - (n)_{q_{ii}} K_{i} \right),
\end{split}
\end{equation}
, we calculate as follows:

\begin{align*}
\psi_{\lambda}^{r_i \mathfrak{q}, \mathfrak{q}} \circ \psi_{\lambda}^{\mathfrak{q}, r_i \mathfrak{q}} \left(v_{\lambda - \rho^{\mathfrak{q}}}^{\mathfrak{q}}\right) 
&= \psi_{\lambda}^{r_i \mathfrak{q}, \mathfrak{q}} \left(E_i^{\operatorname{ord} q_{ii} - 1} v_{\lambda - s_i \rho^{r_i \mathfrak{q}}}^{r_i \mathfrak{q}} \right) \\
&= E_i^{\operatorname{ord} q_{ii} - 1} F_i^{\operatorname{ord} q_{ii} - 1} \left(v_{\lambda - \rho^{\mathfrak{q}}}^{\mathfrak{q}}\right) \\
&= \prod_{n=1}^{\operatorname{ord} q_{ii} - 1} \left( (n)_{q_{ii}^{-1}} \bar{\mathfrak{q}}(\alpha_i, \lambda) - (n)_{q_{ii}} \bar{\mathfrak{q}}(\lambda, \alpha_i)^{-1} \right) v_{\lambda - \rho^{\mathfrak{q}}}^{\mathfrak{q}}.
\end{align*}

Similarly, we have
\[
\psi_{\lambda}^{\mathfrak{q}, r_i \mathfrak{q}} \circ \psi_{\lambda}^{r_i \mathfrak{q}, \mathfrak{q}} \left(v_{\lambda - s_i \rho^{r_i \mathfrak{q}}}^{r_i \mathfrak{q}}\right) 
= \prod_{n=1}^{\operatorname{ord}q_{ii} - 1} \left( (n)_{q_{ii}^{-1}} \bar{\mathfrak{q}}(\lambda, \alpha_i)^{-1} - (n)_{q_{ii}} \bar{\mathfrak{q}}(\alpha_i, \lambda) \right) v_{\lambda - s_i \rho^{r_i \mathfrak{q}}}^{r_i \mathfrak{q}}.
\]

Finally, we observe:
\[
\left( (n)_{q_{ii}^{-1}} \bar{\mathfrak{q}}(\alpha_i, \lambda) - (n)_{q_{ii}} \bar{\mathfrak{q}}(\lambda, \alpha_i)^{-1} \right) = 0
\iff
\left( (n)_{q_{ii}^{-1}} \bar{\mathfrak{q}}(\lambda, \alpha_i)^{-1} - (n)_{q_{ii}} \bar{\mathfrak{q}}(\alpha_i, \lambda) \right) = 0,
\]
by the identity \((n)_q = q^{n-1} (n)_{q^{-1}}\).

\end{proof}

\begin{definition}
We identify the color set of \( RB[\bar{\mathfrak{q}}] \) with \( R_{\bar{\mathfrak{q}}}^+ \).

For \( \lambda \in \mathbb{Z}^\theta \), let \( D_\lambda \) denote the collection of roots \( \alpha \) in \( R_{\bar{\mathfrak{q}}}^+ \) such that
\[
\prod_{n=1}^{\operatorname{ord} q(\alpha, \alpha) - 1} \left( (n)_{q(\alpha, \alpha)^{-1}} \bar{\mathfrak{q}}(\alpha, \lambda) - (n)_{q(\alpha, \alpha)} \bar{\mathfrak{q}}(\lambda, \alpha)^{-1} \right) \neq 0.
\]

We set \( RB[\bar{\mathfrak{q}}, \lambda] := RB[\bar{\mathfrak{q}}] / D_\lambda \).
\end{definition}

The following is exactly the same as in \cite[Corollary 3.27]{Hirota2025}.

\begin{proposition}
The vertex set of \( RB[\bar{\mathfrak{q}}, \lambda] \) can be identified with the isomorphism classes of \(\{ M^{\mathfrak{q}}(\lambda - \rho^{\mathfrak{q}}) \}_{\mathfrak{q} \in V(G[\bar{\mathfrak{q}}])}\).
\end{proposition}

\begin{definition}
Let \( w = \mathfrak{q}_0 c_1 \mathfrak{q}_1 \dots c_t \mathfrak{q}_t \) be a walk in \( RB[\bar{\mathfrak{q}}, \lambda] \). Take \( w_i \in \operatorname{Hom}_{W(G[\bar{\mathfrak{q}}])}(\mathfrak{q}_i, \bar{\mathfrak{q}}) \) for \( i = 0, 1, \dots, t \).

The corresponding composition of nonzero homomorphisms
\[
M^{\mathfrak{q}_0}(\lambda - w_0 \rho^{\mathfrak{q}_0}) 
\xrightarrow{\psi_{\lambda}^{\mathfrak{q}_0, \mathfrak{q}_1}} 
M^{\mathfrak{q}_1}(\lambda - w_1 \rho^{\mathfrak{q}_1}) 
\xrightarrow{\psi_{\lambda}^{\mathfrak{q}_1, \mathfrak{q}_2}} 
\cdots
\]
\[
\cdots 
\xrightarrow{\psi_{\lambda}^{\mathfrak{q}_{t-2}, \mathfrak{q}_{t-1}}} 
M^{\mathfrak{q}_{t-1}}(\lambda - w_{t-1} \rho^{\mathfrak{q}_{t-1}}) 
\xrightarrow{\psi_{\lambda}^{\mathfrak{q}_{t-1}, \mathfrak{q}_t}} 
M^{\mathfrak{q}_t}(\lambda - w_t \rho^{\mathfrak{q}_t}).
\]
is denoted by \( \psi_{\lambda}^{w} \).
\end{definition}

The following theorem is an analogue of odd Verma's theorem in \cite[Theorem 4.9]{Hirota2025}, which we have been aiming for.

\begin{theorem}\label{B3main}
Let \( \lambda \in \mathbb{Z}^\theta \). For a walk \( w \) in \( RB[\bar{\mathfrak{q}}, \lambda] \), the following are equivalent:
\begin{enumerate}
    \item \( \psi_{\lambda}^{w} \neq 0 \).
    \item \( w \) is rainbow.
    \item \( w \) is shortest.
\end{enumerate}
\end{theorem}

\begin{proof}
Using the discussion in this subsection, the argument proceeds exactly as in Section 4 in \cite{Hirota2025}.
\end{proof}

%% file: main.bbl
\begin{thebibliography}{39}
\providecommand{\natexlab}[1]{#1}
\providecommand{\url}[1]{\texttt{#1}}
\expandafter\ifx\csname urlstyle\endcsname\relax
  \providecommand{\doi}[1]{doi: #1}\else
  \providecommand{\doi}{doi: \begingroup \urlstyle{rm}\Url}\fi

\bibitem[Andruskiewitsch and Angiono(2017)]{andruskiewitsch2017finite}
Nicol{\'a}s Andruskiewitsch and Iv{\'a}n Angiono.
\newblock On finite dimensional nichols algebras of diagonal type.
\newblock \emph{Bulletin of Mathematical Sciences}, 7:\penalty0 353--573, 2017.

\bibitem[Andruskiewitsch and Angiono(2008)]{andruskiewitsch2008nichols}
Nicol{\'a}s Andruskiewitsch and Iv{\'a}n~Ezequiel Angiono.
\newblock On nichols algebras with generic braiding.
\newblock In \emph{Modules and comodules}, pages 47--64. Springer, 2008.

\bibitem[Andruskiewitsch et~al.(2018)Andruskiewitsch, Angiono, Mej{\'\i}a, and Renz]{andruskiewitsch2018simple}
Nicol{\'a}s Andruskiewitsch, Iv{\'a}n Angiono, Adriana Mej{\'\i}a, and Carolina Renz.
\newblock Simple modules of the quantum double of the nichols algebra of unidentified diagonal type \textit{ufo}(7).
\newblock \emph{Communications in Algebra}, 46\penalty0 (4):\penalty0 1770--1798, 2018.

\bibitem[Azam et~al.(2015)Azam, Yamane, and Yousofzadeh]{azam2015classification}
Saeid Azam, Hiroyuki Yamane, and Malihe Yousofzadeh.
\newblock Classification of finite-dimensional irreducible representations of generalized quantum groups via weyl groupoids.
\newblock \emph{Publications of the Research Institute for Mathematical Sciences}, 51\penalty0 (1):\penalty0 59--130, 2015.

\bibitem[Bonfert and Nehme(2024)]{bonfert2024weyl}
Lukas Bonfert and Jonas Nehme.
\newblock The weyl groupoids of \( \mathfrak{sl}(m|n) \) and \( \mathfrak{osp}(r|2n) \).
\newblock \emph{Journal of Algebra}, 641:\penalty0 795--822, 2024.

\bibitem[Bouarroudj et~al.(2009)Bouarroudj, Grozman, Leites, et~al.]{bouarroudj2009classification}
Sofiane Bouarroudj, Pavel Grozman, Dimitry Leites, et~al.
\newblock Classification of finite dimensional modular lie superalgebras with indecomposable cartan matrix.
\newblock \emph{SIGMA. Symmetry, Integrability and Geometry: Methods and Applications}, 5:\penalty0 060, 2009.

\bibitem[Chandran et~al.(2018)Chandran, Das, Issac, and van Leeuwen]{chandran2018algorithms}
L~Sunil Chandran, Anita Das, Davis Issac, and Erik~Jan van Leeuwen.
\newblock Algorithms and bounds for very strong rainbow coloring.
\newblock In \emph{LATIN 2018: Theoretical Informatics: 13th Latin American Symposium, Buenos Aires, Argentina, April 16-19, 2018, Proceedings 13}, pages 625--639. Springer, 2018.

\bibitem[Cheng and Wang(2012)]{cheng2012dualities}
Shun-Jen Cheng and Weiqiang Wang.
\newblock \emph{Dualities and representations of Lie superalgebras}.
\newblock American Mathematical Soc., 2012.

\bibitem[Cheng and Wang(2019)]{cheng2019character}
Shun-Jen Cheng and Weiqiang Wang.
\newblock Character formulae in category \( \mathcal{O} \) for exceptional lie superalgebras \( d(2|1; \zeta) \).
\newblock \emph{Transformation Groups}, 24\penalty0 (3):\penalty0 781--821, 2019.

\bibitem[Coggins et~al.(2024)Coggins, Donley~Jr, Gondal, and Krishna]{coggins2024visual}
Terrance Coggins, Robert~W Donley~Jr, Ammara Gondal, and Arnav Krishna.
\newblock A visual approach to symmetric chain decompositions of finite young lattices.
\newblock \emph{arXiv preprint arXiv:2407.20008}, 2024.

\bibitem[Conway and Coxeter(1973)]{conway1973triangulated}
John~H Conway and Harold~SM Coxeter.
\newblock Triangulated polygons and frieze patterns.
\newblock \emph{The Mathematical Gazette}, 57\penalty0 (400):\penalty0 87--94, 1973.

\bibitem[Cuntz(2014)]{cuntz2014frieze}
Michael Cuntz.
\newblock Frieze patterns as root posets and affine triangulations.
\newblock \emph{European Journal of Combinatorics}, 42:\penalty0 167--178, 2014.

\bibitem[Cuntz and Heckenberger(2009{\natexlab{a}})]{cuntz2009weylthree}
Michael Cuntz and Istv{\'a}n Heckenberger.
\newblock Weyl groupoids with at most three objects.
\newblock \emph{Journal of pure and applied algebra}, 213\penalty0 (6):\penalty0 1112--1128, 2009{\natexlab{a}}.

\bibitem[Cuntz and Heckenberger(2009{\natexlab{b}})]{cuntz2009weyltwo}
Michael Cuntz and Istv{\'a}n Heckenberger.
\newblock Weyl groupoids of rank two and continued fractions.
\newblock \emph{Algebra \& Number Theory}, 3\penalty0 (3):\penalty0 317--340, 2009{\natexlab{b}}.

\bibitem[Etingof et~al.(2015)Etingof, Gelaki, Nikshych, and Ostrik]{etingof2015tensor}
Pavel Etingof, Shlomo Gelaki, Dmitri Nikshych, and Victor Ostrik.
\newblock \emph{Tensor categories}, volume 205.
\newblock American Mathematical Soc., 2015.

\bibitem[Gorelik et~al.(2022)Gorelik, Hinich, and Serganova]{gorelik2022root}
Maria Gorelik, Vladimir Hinich, and Vera Serganova.
\newblock Root groupoid and related lie superalgebras.
\newblock \emph{arXiv preprint arXiv:2209.06253}, 2022.

\bibitem[Gorelik et~al.(2023)Gorelik, Hinich, and Serganova]{gorelik2023matsumoto}
Maria Gorelik, Vladimir Hinich, and Vera Serganova.
\newblock Matsumoto theorem for skeleta.
\newblock \emph{arXiv preprint arXiv:2310.13507}, 2023.

\bibitem[Heckenberger(2010)]{heckenberger2010lusztig}
I~Heckenberger.
\newblock Lusztig isomorphisms for drinfel'd doubles of bosonizations of nichols algebras of diagonal type.
\newblock \emph{Journal of algebra}, 323\penalty0 (8):\penalty0 2130--2182, 2010.

\bibitem[Heckenberger(2006)]{heckenberger2006weyl}
Istv{\'a}n Heckenberger.
\newblock The weyl groupoid of a nichols algebra of diagonal type.
\newblock \emph{Inventiones mathematicae}, 164\penalty0 (1):\penalty0 175--188, 2006.

\bibitem[Heckenberger(2009)]{heckenberger2009classification}
Istv{\'a}n Heckenberger.
\newblock Classification of arithmetic root systems.
\newblock \emph{Advances in Mathematics}, 220\penalty0 (1):\penalty0 59--124, 2009.

\bibitem[Heckenberger and Schneider(2020)]{heckenberger2020hopf}
Istv{\'a}n Heckenberger and Hans-J{\"u}rgen Schneider.
\newblock \emph{Hopf algebras and root systems}, volume 247.
\newblock American Mathematical Soc., 2020.

\bibitem[Heckenberger and Yamane(2008)]{heckenberger2008generalization}
Istv{\'a}n Heckenberger and Hiroyuki Yamane.
\newblock A generalization of coxeter groups, root systems, and matsumoto’s theorem.
\newblock \emph{Mathematische Zeitschrift}, 259:\penalty0 255--276, 2008.

\bibitem[Hirota(2025)]{Hirota2025}
Shunsuke Hirota.
\newblock Odd verma's theorem.
\newblock \emph{arXiv preprint arXiv:2502.14274}, 2025.

\bibitem[Humphreys(1992)]{humphreys1992reflection}
James~E Humphreys.
\newblock \emph{Reflection groups and Coxeter groups}.
\newblock Number~29. Cambridge university press, 1992.

\bibitem[Inoue and Yamane(2023)]{inoue2023hamiltonian}
Takato Inoue and Hiroyuki Yamane.
\newblock Hamiltonian cycles for finite weyl groupoids.
\newblock \emph{arXiv preprint arXiv:2310.12543}, 2023.

\bibitem[Kac and Weisfeiler(1971)]{kac_weisfeiler_1971}
V.~Kac and B.~Weisfeiler.
\newblock Exponentials in lie algebras of characteristic p.
\newblock \emph{Math. USSR Izv.}, 5:\penalty0 777--803, 1971.

\bibitem[Kac(1990)]{kac1990infinite}
Victor~G Kac.
\newblock \emph{Infinite-dimensional Lie algebras}.
\newblock Cambridge university press, 1990.

\bibitem[Kharchenko(2015)]{kharchenko2015quantum}
Vladislav Kharchenko.
\newblock Quantum lie theory.
\newblock \emph{Lecture Notes in Mathematics}, 2150, 2015.

\bibitem[Laugwitz and Sanmarco(2023)]{laugwitz2023finite}
Robert Laugwitz and Guillermo Sanmarco.
\newblock Finite-dimensional quantum groups of type super a and non-semisimple modular categories.
\newblock \emph{arXiv preprint arXiv:2301.10685}, 2023.

\bibitem[Li et~al.(2013)Li, Shi, and Sun]{li2013rainbow}
Xueliang Li, Yongtang Shi, and Yuefang Sun.
\newblock Rainbow connections of graphs: A survey.
\newblock \emph{Graphs and combinatorics}, 29:\penalty0 1--38, 2013.

\bibitem[Lusztig(1990)]{lusztig1990quantum}
George Lusztig.
\newblock Quantum groups at roots of 1.
\newblock \emph{Geometriae Dedicata}, 35\penalty0 (1):\penalty0 89--113, 1990.

\bibitem[Lusztig(2010)]{lusztig2010introduction}
George Lusztig.
\newblock \emph{Introduction to quantum groups}.
\newblock Springer Science \& Business Media, 2010.

\bibitem[Musson(2012)]{musson2012lie}
Ian~Malcolm Musson.
\newblock \emph{Lie superalgebras and enveloping algebras}, volume 131.
\newblock American Mathematical Soc., 2012.

\bibitem[Serganova(2011)]{serganova2011kac}
Vera Serganova.
\newblock Kac--moody superalgebras and integrability.
\newblock \emph{Developments and trends in infinite-dimensional Lie theory}, pages 169--218, 2011.

\bibitem[Serganova(2017)]{serganova2017representations}
Vera Serganova.
\newblock Representations of lie superalgebras.
\newblock \emph{Perspectives in Lie theory}, pages 125--177, 2017.

\bibitem[Stanley et~al.(2012)]{stanley2012topics}
Richard~P Stanley et~al.
\newblock Topics in algebraic combinatorics.
\newblock \emph{Course notes for Mathematics}, 192:\penalty0 13, 2012.

\bibitem[Vay(2018)]{vay2018hopf}
Cristian Vay.
\newblock On hopf algebras with triangular decomposition.
\newblock \emph{arXiv preprint arXiv:1808.03799}, 2018.

\bibitem[Vay(2023)]{vay2023linkage}
Cristian Vay.
\newblock Linkage principle for small quantum groups.
\newblock \emph{arXiv preprint}, 2023.
\newblock arXiv:2310.00103.

\bibitem[Yamane(2021)]{yamane2021hamilton}
Hiroyuki Yamane.
\newblock Hamilton circuits of cayley graphs of weyl groupoids of generalized quantum groups.
\newblock \emph{arXiv preprint arXiv:2103.16126}, 2021.

\end{thebibliography}
